\documentclass[11pt,a4paper]{article}

\usepackage{comment}
\usepackage{graphicx}
\usepackage{amsmath}
\usepackage{amsfonts}
\usepackage{amssymb}
\usepackage{enumerate}

\usepackage{float}
\newfloat{Test}{htbp}{lot}
\floatname{Test}{Test}

\usepackage{lscape}
\usepackage{longtable}
\usepackage{rotating}
\usepackage{color}

\usepackage{tabularx}
\usepackage{array}
\newcolumntype{Y}{>{\centering\arraybackslash}X}

\usepackage{url}
\usepackage{rotating}
\usepackage{algpseudocode}
\usepackage[ruled]{algorithm}
\usepackage{algcompatible}
\usepackage{pgf,tikz}
\usepackage{subfig}
\usepackage[titletoc]{appendix}
\usepackage{wrapfig}
\usepackage{adjustbox}
\usepackage{mathrsfs}
\usepackage{pgfplots}
\usetikzlibrary{arrows}
\usepackage{color}

\usepackage{bbm}

\usepackage{booktabs}
\usepackage{multirow}
\usepackage{mathtools}
\usepackage{makecell}

\usepackage[sort,numbers]{natbib}

\usepackage{eqparbox}%

\numberwithin{table}{section}
\numberwithin{figure}{section}
\numberwithin{equation}{section}%

\usepackage[symbol]{footmisc}%

\newtheorem{theorem}{Theorem}[section]

\newtheorem{lemma}[theorem]{Lemma}

\newtheorem{assumption}{Assumption}[section]

\newenvironment{proof}[1][Proof]{\noindent \textbf{#1.} }{\hfill$\Box$\par\medskip}

\newcommand{\beqn}[1]{\begin{equation}\label{#1}}
\newcommand{\eeqn}{\end{equation}}

\newcommand{\Lf}{L}

\definecolor{darkgreen}{rgb}{0,0.6,0}
\definecolor{aau2}{rgb}{0.0, 0.5, 0.69}
\definecolor{aau3}{rgb}{0.0, 0.53, 0.74}
\definecolor{aau4}{rgb}{0.0, 0.48, 0.65}
\definecolor{aau5}{rgb}{0.0, 0.45, 0.73}
\definecolor{rsap}{RGB}{130, 36, 51}
\definecolor{gsap}{RGB}{112, 164, 137}

\definecolor{tud}{rgb}{0.43,0.73,0.11}
\definecolor{verde}{rgb}{0.33,0.53,0.11}

\definecolor{ttffqq}{rgb}{0.0, 0.48, 0.65} 
\definecolor{ffqqqq}{rgb}{0.0, 0.5, 0.69} 

\usetikzlibrary{arrows}

\tikzstyle{decision} = [diamond, draw, fill=blue!20,
text width=4.5em, text badly centered, node distance=3cm, inner sep=0pt]
\tikzstyle{block} = [rectangle, draw, fill=blue!20,
text centered, rounded corners, minimum height=4em]
\tikzstyle{line} = [draw, -latex']
\tikzstyle{cloud} = [draw, ellipse,fill=red!20, node distance=3cm,
minimum height=2em]
\tikzstyle{cloud2} = [draw, ellipse,fill=green!20, node distance=3cm,
minimum height=2em]

\pgfplotsset{compat=1.18}
\begin{document}
	
\title{Non-monotone direct-search methods for deterministic and stochastic derivative-free optimization}

	\author{
		A. Ding\thanks{Department of Industrial and Systems Engineering, Lehigh University, Bethlehem, PA 18015-1582, USA ({\tt and523@lehigh.edu}).}
		\and
        Trang H. Tran \thanks{Department of Mathematics, Emory University, Atlanta, GA 30322-1005, USA ({\tt htran31@emory.edu}). This work was done when the author was a postdoctoral researcher at Lehigh University.}
	\and
	L. N. Vicente\thanks{Department of Industrial and Systems Engineering, Lehigh University, Bethlehem, PA 18015-1582, USA ({\tt lnv@lehigh.edu}).}
	}
	
	\maketitle

\begin{abstract}

In derivative-free optimization (DFO), one minimizes functions for which the gradient is unavailable or expensive to compute. In many applications, objective function values and gradients are noisy due to simulations or system randomness. A class of standard direct-search methods for DFO accept a trial point when it decreases the objective function by an amount proportional to the squared stepsize. However, when applied to complex landscapes, such a requirement may trap the algorithm in a neighborhood of sub-optimal solutions. We study a non-monotone direct-search alternative where the trial function value is compared with the largest objective function obtained through the $M$ most recent distinct iterates. This max-$M$ non-monotone condition permits temporary increases in the objective function and can help navigate narrow curved valleys; however, its theoretical analysis is significantly more challenging due to the lack of monotonic decrease. 
In this paper, we develop a comprehensive complexity theory for the max-$M$ non-monotone direct-search in both deterministic and stochastic DFO problems. 
For deterministic objectives, we establish a worst-case iteration bound for a complete poll based on a positive spanning set and an expected iteration bound for a probabilistic-descent poll. We then analyze a stochastic variant using independent function estimates and show the expected iteration complexity under tail-bound assumptions of the stochastic errors. 
All three results have the standard complexity of $\mathcal{O}(\epsilon^{-2})$, which matches the iteration complexity of monotone direct-search methods. Our theory is enabled by a new family of merit functions that correct the stored objective values by ordered multiples of the squared stepsize, together with a renewal-reward stopping-time argument for the probabilistic methods. Numerical experiments on CUTEst problems indicate that when comparing with monotone direct search, the max-$M$ non-monotone method is particularly helpful on problems exhibiting negative curvature. 
\end{abstract}


\section{Introduction}

In this paper, we consider unconstrained optimization problems of the form
\begin{align*}
\min_{x\in\mathbb{R}^n}f(x),
\end{align*}
whose derivatives are assumed to be unavailable and costly to approximate. Besides, the objective function $f$ is assumed to be $L$-smooth and bounded below by $f^*$. We will first consider the case in which optimization algorithms have access to a zeroth-order oracle $f(x)$. Then, we will consider the case in which optimization algorithms have access to a stochastic zeroth-order oracle $F(x,\xi)$, which provides noisy observations of its deterministic counterpart $f(x)$ (and where~$\xi$ denotes a random estimator). The absence of derivative information is a common feature in many simulation-based optimization problems and constitutes the primary focus of derivative-free optimization (DFO). In this setting, access to the objective function $f$ is limited to evaluations via a zeroth-order oracle, which are often computationally expensive. Consequently, the central objective of DFO is to obtain high-quality solutions with few such oracle queries.

Trust-region, direct-search, and line-search methods constitute three major classes of algorithms in DFO. Trust-region methods construct local surrogate models of $f$, which are then used to generate trial steps and determine their acceptance. In contrast, direct-search methods probe the search space using a set of directions or polling points without relying on explicit models. Bridging these two paradigms, line-search methods approximate descent directions, typically via finite differences, and perform searches along these directions. For a comprehensive overview of these classes of DFO methods, we refer the reader to, e.g.,~\cite{conn2009introduction, larson2019derivative,audet2017derivative}. 
A representative 
example within direct-search methods is the probabilistic descent approach proposed in~\cite{gratton2015direct}. Unlike other classical direct-search methods that rely on a positive spanning set (see, e.g.,~\cite{conn2009introduction}) 
to guarantee a descent direction (requiring at least $n+1$ function evaluations), this method incorporates randomness to obtain a descent direction with a sufficiently large probability using only one or two function evaluations. 
It is shown in~\cite{gratton2015direct} that, for deterministic objective functions, the method achieves a zeroth-order oracle complexity of $\mathcal{O}(n/\epsilon^2)$ to find an $\epsilon$-approximate first-order stationary point $x$ defined by $\|\nabla f(x)\|\leq\epsilon$.

It is recognized in~\cite{grippo1986nonmonotone,chamberlain2009watchdog,dennis1996numerical} that enforcing the monotonicity of the function values may considerably slow the convergence of optimization algorithms, 
particularly in the presence of narrow curved valleys. To alleviate this phenomenon, a non-monotone technique was proposed in~\cite{grippo1986nonmonotone} for line-search algorithms, with reported favorable numerical results. Roughly speaking, such a non-monotone line search relaxes the Armijo condition (see, e.g.,~\cite{nocedal2006numerical}) 
and allows some non-monotonicity by using the maximum of objective function values at $M+1$ past and current iterate points, as follows
\begin{align*}
f(x_{k}+\alpha_{k}d_{k})\leq\max_{0\leq j\leq M}f(x_{k-j})+\frac{1}{2}\alpha_{k}\nabla f(x_{k})^\top d_{k},
\end{align*}
where $d_{k}$ is a descent direction, $\nabla f(x_{k})^\top d_{k}<0$, and $\alpha_{k}>0$ is a stepsize parameter. From this design, a non-monotone line search~\cite{zhang2004nonmonotone} may avoid creeping along the bottom of a narrow curved valley and can generally obtain better numerical results.

\subsection{Literature review}
We first review some theoretical results that have been obtained for non-monotone line search. In the original paper~\cite{grippo1986nonmonotone}, non-monotone line search was proven to converge globally, meaning that it converges for an arbitrary starting iterate point $x_{0}$, for smooth non-convex functions. 
Later, in~\cite{dai2002nonmonotone}, 
it was shown to exhibit a linear rate for uniformly convex functions. One variant using a moving average of past objective function values instead of their maximum was proposed and investigated in~\cite{zhang2004nonmonotone}, where it was shown to converge globally for smooth non-convex functions and to exhibit linear rates for strongly convex functions. 
For smooth non-convex functions, non-monotone line search was proven in~\cite{cartis2015worst} to find an $\epsilon$-approximate first-order stationary point in at most $\mathcal{O}(\epsilon^{-2})$ function and gradient evaluations.

Another interesting application setting for non-monotone algorithms is when deterministic noise 
occurs in the objective function. The deterministic noise in the objective function may destroy the benign geometry of the objective function 
and cause many spurious local minima, even when the algorithm is still far from stationarity. To handle this problem, especially when noise is bounded, the authors in~\cite{berahas2019derivative} add a constant non-monotonicity term $\epsilon_{f}>0$ to the Armijo condition as follows:
\begin{align*}
f(x_{k}+\alpha_{k}d_{k})\leq f(x_{k})+\frac{1}{2}\alpha_{k}\nabla f(x_{k})^\top d_{k}+\epsilon_{f}.
\end{align*}
Similarly, to handle bounded deterministic noises, the authors in~\cite{cao2024first} consider the following acceptance condition in trust-region methods
\begin{align*}
\frac{f(x_{k})-f(x_{k}+s_{k})+r}{m_{k}(x_{k})-m_{k}(x_{k}+s_{k})}\geq\eta_{1},
\end{align*}
where $\eta_{1}>0$ is a parameter, $s_{k}$ is a candidate point, $m_{k}$ is a local model, and $r>0$ is also a constant non-monotonicity parameter term. 
Different variants and analyzes following this lead can be found, for example, in~\cite{berahas2021global,larson2025novel}.

Now let us also review what has been proposed so far in non-monotone (derterministic) methods for DFO. The first non-monotone DFO algorithm known to us is a line-search algorithm that was proposed and shown in~\cite{diniz2008derivative} to converge globally. At the beginning of one iteration in their linesearch algorithm, a set of search directions $D_{k}$ is computed, and a stepsize $\alpha$ is set to $1$. Then the algorithm repeatedly shrinks the stepsize until the following sufficient decrease condition is satisfied for one of the directions $d$ in $D_{k}$
\begin{align*}
f(x_{k}+\alpha d)\leq\max_{0\leq j\leq M}f(x_{k-j})+\eta_{k}-\alpha^2\beta_{k},
\end{align*}
where $\{\eta_{k}>0\}$ satisfies $\sum_{k=0}^{\infty}\eta_{k}<\infty$
and $\{\beta_{k}>0\}$ is selected in a way that for all infinite subsets of indices $K\subset\mathbb{N}$, one has $\lim_{k\in K}\beta_{k}=0\Rightarrow\lim_{k\in K}\nabla f(x_{k})=0$. Therefore, at the end of each iteration, under appropriate assumptions, a distinct $x_{k+1}$ is always found. Direct search and line search are very similar algorithms, but the stepsize in direct search is maintained adaptively across different iterations. This distinction then requires a slightly different analysis for these two algorithms.
Other non-monotone derivative-free algorithms include a parallel and sequential algorithm~\cite{garcia2013sequential} using a space decomposition scheme for box constrained optimization problems, an inexact line-search algorithm~\cite{grippo2015class} constructing search directions by combining
coordinate rotations with simplex gradients, and a direct-search algorithm~\cite{garcia2020non} for linearly constrained minimization problems.
Although these derivative-free non-monotone methods have generally been shown to converge globally, their iteration complexity results have yet to be established.
While finishing this paper we became aware of a new paper~\cite{NIMA} establishing a worst-case complexity result for non-monotone direct search, where $f(x_k)$ in the sufficient decrease condition is replaced by $f(x_k) + \eta_k$, with $\eta_k \ge 0$ chosen to satisfy $\sum_{k=0}^{\infty}\eta_{k}<\infty$. This condition is different from the max-$M$ acceptance rule that it requires a user-specified sequence~$\eta_k$ to control the non-monotonicity at each iteration $k$. In practice, parameter tuning is often needed to find the best values for $\eta_k$ as these values usually depends on the scale of the objective function and the problem. 

\subsection{Our contributions}



In this paper, we develop and analyze max-$M$ non-monotone direct-search methods for deterministic and stochastic DFO. We replace the usual sufficient decrease condition
\begin{align*}
f(x_{k}+\delta_{k}d_k)\leq f(x_{k})-c\delta_k^2,
\end{align*}
with the following acceptance rule 
\begin{align*}
f(x_{k}+\delta_{k}d_{k})\leq\max\left(f(x_{succ,k}^{M-1}),\ldots,f(x_{succ,k}^{1}),f(x_{k})\right)-c\delta_k^2,
\end{align*}
where $c>0$, $\delta_k>0$ is the adaptive stepsize, $d_k$ is a search direction, 
$x_{succ,k}^{i}$ is the $i$-th latest successful iterate from $x_{k}$, and $M > 1$ is a chosen memory length. 
The method may therefore accept a trial point whose objective value exceeds $f(x_k)$, but it still requires a decrease of order $\delta_k^2$ compared to the largest function value in a window of $M$ past iterates.

For deterministic objectives, we show  theoretical analysis for direct-search methods with this acceptance rule: one based on a positive spanning set and one based on probabilistic descent directions. For the former, we establish a worst-case bound on the number of iterations required to produce an iterate $x_k$ satisfying $\|\nabla f(x_k)\|\leq\epsilon$, while for the latter, we establish an expected iteration bound. Both analysis recover the standard $\mathcal{O}(\epsilon^{-2})$ dependence on the stationary tolerance~$\epsilon$ for monotone direct-search methods.
In addition, we propose a stochastic max-$M$ direct-search algorithm where acceptance test uses independent stochastic estimates of the objective values at $x_k + \delta_k d_k$, $x_k$, and $x_{succ,k}^{i}$ for $i= 1, \dots, M-1$. Under conditional probabilistic assumptions on the sampling error and on a probabilistic descent direction, we show an $\mathcal{O}(\epsilon^{-2})$ bound on the expected number of iterations for the proposed stochastic method. 
Finally, we show a competitive, sometimes better performance of deterministic and stochastic max-$M$ algorithms compared to their monotone counterparts on two sets of CUTEst problems. 

To the best of our knowledge, this is the first work that analyzes the non-monotone method in stochastic DFO, and also the first work that analyzes the max-$M$ acceptance rule for direct search.
The main technical challenge in our analysis is the construction of a merit function adapted to the max-$M$ non-monotone decrease condition. The most natural, naive merit function consists of the maximum objective value among $M$ prior iterates and a multiple of the squared stepsize, however, it may not decrease after a successful iteration because the stepsize is increased. To overcome this difficulty, we subtract ordered terms of the squared stepsize from the objective values before taking their maximum. These corrections are carefully chosen to account both for the shift of merit function per iteration and for the change in the stepsize. Most importantly, this merit function allows for a decrease which only depends on the acceptance condition at the current iteration. This local dependence makes such merit function feasible for the analysis of the subsequent stochastic algorithm. 

\section{Merit function design for max-$M$ non-monotone direct search}\label{sec:merit_functions}

Our convergence theory is based on the development of an appropriate merit (Lyapunov) function, and we would like to discuss the main properties of such functions that guided our efforts. Merit functions are commonly seen in constrained optimization~\cite{nocedal2006numerical} and are used in the analyses of many algorithms for continuous optimization of non-linear functions. They are designed to evaluate the progress of an iterative algorithm. They are usually bounded from below and expected to decrease by a non-negligible amount at each iteration of an algorithm before a stationary point is found. Such properties of merit functions allow us to use them for the derivation of asymptotic global convergence or iteration complexity bounds. As optimization algorithms become more advanced and complicated, merit functions provide an intuitive and simplified approach to analyze these methods (e.g., \cite{blanchet2019convergence} provides an framework for analyzing stochastic objective functions).


For monotone direct search, the objective value itself is a natural progress measure because every successful iteration satisfies
\begin{align*}
f(x_{k+1})\leq f(x_k)-c\delta_k^2.
\end{align*}
This argument is unavailable for max-$M$ direct search. The accepted point is compared with the largest objective value among the current and $M-1$ most recent distinct iterates, and hence $f(x_{k+1})$ may exceed $f(x_k)$. 

The most immediate and natural candidate for a merit function is
\begin{align*}
\widehat\Phi_k=\max\left(f(x_{succ,k}^{M-1}),\ldots,f(x_{succ,k}^{1}),f(x_k)\right)-f^*+\eta\delta_k^2,
\end{align*}
where $\eta>0$ is a constant parameter. This candidate is bounded below, but it does not necessarily decrease at each iteration.
In fact, suppose that iteration $k$ is successful and that $$\max\left(f(x_{succ,k}^{M-2}),\ldots,f(x_{succ,k}^{1}),f(x_{k})\right)\geq f(x_{succ,k}^{M-1}),$$
then we have 
\begin{align*}
\widehat\Phi_{k+1}  =&\max{\left(f(x_{succ,k+1}^{M-1}),\ldots,f(x_{succ,k+1}^{1}),f(x_{k+1})\right)}+\eta\delta_{k+1}^2 - f^*\\
=&\max\left(f(x_{succ,k}^{M-2}),\ldots,f(x_{k}),f(x_{k+1})\right)+\eta\gamma^2\delta_{k}^2- f^*\\
\geq&\max\left(f(x_{succ,k}^{M-2}),\ldots,f(x_{k})\right)+\eta\gamma^2\delta_{k}^2- f^*\\
\geq&\max\left(f(x_{succ,k}^{M-1}),\ldots,f(x_{succ,k}^{1}),f(x_{k})\right)+\eta\gamma^2\delta_{k}^2- f^*.
\end{align*}
Since $\gamma>1$, $\widehat\Phi_k$ increases by at least $\eta(\gamma^2-1)\delta_{k}^2$ in this case. Therefore, $\widehat\Phi_k$ does not offer a decrease, and we need a new merit function design for the max-$M$ non-monotone condition. 

Let $0<c_1<\cdots<c_{M-1}<\eta$ be an increasing sequence of constants. 
Our main design assigns different stepsize correction $c_{i}\delta_k^2$ to each prior objective value $f(x_{succ,k}^{i})$ as follows:
\begin{align}\label{meritd}
\Phi_k={}&\max\left(f(x_{succ,k}^{M-1})-c_{M-1}\delta_k^2,\ldots,
f(x_{succ,k}^{1})-c_1\delta_k^2,f(x_k)\right) -f^*+\eta\delta_k^2,
\end{align}
This merit function is non-negative, and the oldest stored value receives the largest correction. We note that $\Phi_{k}$ reduces to $\widehat\Phi_k$ when $c_{1}=\cdots=c_{M-1}=0$.
On a successful iteration, the memory window shifts, the current iterate becomes a past iterate, and the stepsize is enlarged. The ordered corrections $c_i > 0$ are carefully chosen so that the quantity $c_i\gamma^2-c_{i-1}$ is positive and the change in the merit function generates a decrease. On an unsuccessful iteration, the past points remain unchanged and the stepsize contracts; hence choosing $\eta$ sufficiently large relative to $c_{M-1}$ makes the decrease in $\eta\delta_k^2$ dominate the change in the corrected maximum. Balancing these two cases yields constants for which
\begin{align*}
\Phi_k-\Phi_{k+1}\geq \nu\delta_k^2,
\end{align*}
for some constant $\nu>0$.

For the direct-search method based on a positive spanning set, we utilize a simpler merit function, as this algorithm allows for a lower bound on the stepsize. As a result, the dynamic stepsize $\delta_k$ was replaced by a constant threshold $\delta_\epsilon$. In contrast, for the latter cases where the search direction is random and the function potentially contains noise, the stepsize does not admit a lower bound, which makes the analyis of the merit function much more challenging. For stochastic DFO, our analysis obtains a bound on the expected decrease of the merit function, instead of a deterministic bound.

\section{Non-monotone direct search for deterministic functions}
\subsection{Max-$M$ non-monotone direct search with positive spanning set}
The first method in our consideration (presented in Algorithm~\ref{algo1pss}) replaces the sufficient decrease condition in standard direct search with positive spanning set~\cite{conn2009introduction} by the max-$M$ non-monotone condition. Without loss of generality, let $\mathcal{D}$ be a positive spanning set with cosine measure $1/\sqrt{n}$, e.g., $\mathcal{D}=\{\pm e_{1},\ldots,\pm e_{n}\}$. The algorithm searches for a direction $d_k \in \mathcal{D}$ such that $f(x_k + \delta_k d_k)$ satisfies the decrease condition~\eqref{sdcdpss}.
If the algorithm finds a successful direction, the polling point becomes the next iterate and the stepsize is enlarged by $\gamma$; if no direction in $\mathcal D$ is successful, the iterate is retained and the stepsize is reduced by $\theta$.


\begin{algorithm}[H]
\caption{Max-$M$ Non-monotone Direct Search Based on Positive Spanning Set}\label{algo1pss}
\begin{algorithmic}[1]
\STATE{Initialization. Choose integer $M>1$, $c>0$, $x_0$, $\delta_0$, $\theta\in(0,1)$, $\gamma\in(1,\infty)$.}
\FOR{$k=0,1,\ldots$}
\STATE{For $i\in\mathbb{N}^{+}$, if there is an $i$-th latest successful iterate point from $x_{k}$, then record it with~$x_{succ,k}^{i}$. Otherwise let $x_{succ,k}^{i}$ be $x_{0}$.}
\IF{there exists $d_{k}\in\mathcal{D}$ such that
\begin{align}\label{sdcdpss}
\max\left(f(x_{succ,k}^{M-1}),\ldots,f(x_{succ,k}^{1}),f(x_{k})\right)-f(x_{k}+\delta_{k}d_k)\geq c\delta_k^2,
\end{align}
}
\STATE{Set $x_{k+1}=x_k+\delta_k d_k$ and $\delta_{k+1}=\gamma\delta_{k}$. This iteration is successful.}
\ELSE
\STATE{Set $x_{k+1}=x_k$ and $\delta_{k+1}=\theta\delta_{k}$. This iteration is not successful.}
\ENDIF
\ENDFOR
\end{algorithmic}
\end{algorithm}

The goal of our analysis is to bound the number of iterations $T_\epsilon$ defined as follows
\begin{align*}
T_\epsilon=\inf\{k\geq0:\|\nabla f(x_k)\|\leq\epsilon\}.
\end{align*}
We first assume that the objective function has a lower bound and its gradient is Lipschitz smooth, which are standard assumptions in first-order complexity analyses of smooth, non-convex direct-search methods.
\begin{assumption}\label{a21}
The objective function $f(x)$ is bounded below by $f^*$.
\end{assumption}

\begin{assumption}\label{a22}
The gradient of the objective function $f(x)$ is Lipschitz continuous (with constant $\Lf > 0$).
\end{assumption}

As the objective is Lipschitz smooth, we show (using well known arguments) that before the gradient norm approaches $\epsilon$, there exists a threshold $\delta_{\epsilon}$ that any stepsize below $\delta_{\epsilon}$ produces a successful iteration.


\begin{lemma}\label{lemma2.1}
Let Assumption~\ref{a22} hold and let $\delta_{\epsilon}=\frac{2\epsilon}{\sqrt{n}(\Lf+2c)}$. For Algorithm~\ref{algo1pss}, if $\delta_{k}\leq\delta_{\epsilon}$ and $k<T_{\epsilon}$, then iteration $k$ is successful.
\end{lemma}

\begin{proof}
It is well known~(see, for example,~\cite[Lemma 1.2.3]{nesterov2013introductory}) that Assumption~\ref{a22} implies for any $x,y$ from $\mathbb{R}^n$ that
\begin{align*}
f(x)-f(y)\geq\nabla f(x)^\top(x-y)-\frac{\Lf}{2}\|x-y\|^2.
\end{align*}
After substituting $x=x_k$ and $y=x_k+\delta_kd_k$, since $\|d_{k}\|=1$, we obtain
\begin{align}\label{p1}
f(x_k)-f(x_k+\delta_kd_k)&\geq-\nabla f(x_k)^\top (\delta_kd_k)-\frac{\Lf}{2}\delta_k^2.
\end{align}
Since the cosine measure of $\mathcal{D}$ is $1/\sqrt{n}$, by definition of cosine measure, we have
\begin{align}\label{cm}
\max_{d_{k}\in\mathcal{D}}\frac{-\nabla f(x_{k})^\top d_{k}}{\|\nabla f(x_{k})\| \|d_{k}\|}\geq\frac{1}{\sqrt{n}}.
\end{align}
Note that $k<T_{\epsilon}$ implies $\|\nabla f(x_{k})\|>\epsilon$. Together with~(\ref{cm}) and $\|d_{k}\|=1$, there exists $d_{k}\in\mathcal{D}$ such that
\begin{align}\label{p2}
-\nabla f(x_k)^\top (\delta_kd_k)\geq\frac{1}{\sqrt{n}}\|\nabla f(x_{k})\|\|d_{k}\|\delta_{k}>\frac{\epsilon\delta_{k}}{\sqrt{n}}.
\end{align}
After combining~(\ref{p1}) and~(\ref{p2}), we have: when $\delta_{k}\leq\delta_{\epsilon}=\frac{2\epsilon}{\sqrt{n}(\Lf+2c)}$ and $k<T_{\epsilon}$, there exists $d_{k}\in\mathcal{D}$ such that
\begin{align*}
\max\left(f(x_{succ,k}^{M-1}),\ldots,f(x_{succ,k}^{1}),f(x_{k})\right)-f(x_{k}+\delta_{k}d_k)&\geq f(x_k)-f(x_k+\delta_kd_k)\\
&\geq\frac{\epsilon\delta_{k}}{\sqrt{n}}-\frac{\Lf}{2}\delta_k^2\geq c\delta_{k}^2,
\end{align*}
which means iteration $k$ is successful.
\end{proof}

From the result of Lemma~\ref{lemma2.1}, we show that all the stepsizes in Algorithm~\ref{algo1pss} admit a lower bound as follows.
\begin{lemma}\label{lemma2.2}
Let Assumption~\ref{a22} hold. In Algorithm~\ref{algo1pss}, suppose that $\delta_{\epsilon}\leq\delta_{0}$ and $k\leq T_{\epsilon}$. We have
\begin{align*}
\delta_{k}\geq\theta\delta_{\epsilon}.
\end{align*}
\end{lemma}

\begin{proof}
We prove this by induction. For $k=0$, we have $\delta_{0}\geq\delta_{\epsilon}\geq\theta\delta_{\epsilon}$, and the conclusion holds. We assume for induction that $t\leq T_{\epsilon}$ and $\delta_{t}\geq\theta\delta_{\epsilon}$ for some arbitrary $t\geq 0$. Then it suffices to prove that $t+1\leq T_{\epsilon}$ implies $\delta_{t+1}\geq\theta\delta_{\epsilon}$. We prove this by cases. If $\delta_{t}\geq\delta_{\epsilon}$, then by the algorithm mechanism, we have $\delta_{t+1}\geq\theta\delta_{\epsilon}$. Otherwise, we have $\theta\delta_{\epsilon}\leq\delta_{t}<\delta_{\epsilon}$. Since $t+1\leq T_{\epsilon}$ implies $t<T_{\epsilon}$, using Lemma~\ref{lemma2.1} shows that iteration $t$ is successful. Therefore, the algorithm mechanism indicates that $\delta_{t+1}=\gamma\delta_{t}>\delta_{t}\geq\theta\delta_{\epsilon}$, which concludes the proof.
\end{proof}
We are interested in the number of successes before iteration $k$, denoted by $N_{succ,k}$. 
The stepsize mechanism of Algorithm~\ref{algo1pss} implies
\begin{align*}
\delta_{k}=\delta_{0}\gamma^{N_{succ,k}}\theta^{k-N_{succ,k}}.
\end{align*}
The following lemma shows a relation between $N_{succ,k}$ and the total number of iteration $k$.
\begin{lemma}\label{lemma2.3}
Let Assumption~\ref{a22} hold. In Algorithm~\ref{algo1pss}, suppose that $\delta_{\epsilon}\leq\delta_{0}$ and $k\leq T_{\epsilon}$. We have
\begin{align*}
k\leq\frac{\ln{\frac{\delta_{0}}{\theta\delta_{\epsilon}}}+N_{succ,k}\ln{\frac{\gamma}{\theta}}}{\ln{\frac{1}{\theta}}}.
\end{align*}
\end{lemma}

\begin{proof}
From Lemma~\ref{lemma2.2}, we have $\delta_{k}\geq\theta\delta_{\epsilon}$. The result is obtained after using the fact that  $\delta_{k}=\delta_{0}\gamma^{N_{succ,k}}\theta^{k-N_{succ,k}}$ and rearranging.
\end{proof}

Let the merit function for Algorithm~\ref{algo1pss} be defined as 
$$\phi_{k}=\max\left(f(x_{succ,k}^{M-1})-c_{M-1}\theta^2\delta_{\epsilon}^2,\ldots,f(x_{succ,k}^{1})-c_{1}\theta^2\delta_{\epsilon}^2,f(x_{k})\right)-f^{*},$$
where $0<c_1<\cdots<c_{M-1}<\eta$. 
Note that this function is slightly different from $\Phi_k$ in~\eqref{meritd} which uses the squared stepsize $\delta_k^2$. 
Now we are ready to present the decrease lemma for $\phi_{k}$ after a successful iteration.
\begin{lemma}\label{lemma2.4}
Let Assumption~\ref{a22} hold. Let $c_{i}=\frac{i}{M}c$, $1\leq i\leq M-1$. In Algorithm~\ref{algo1pss}, suppose that $\delta_{\epsilon}\leq\delta_{0}$ and $k\leq T_{\epsilon}$. From the satisfaction of the sufficient decrease condition~(\ref{sdcdpss}) at iteration $k$, we have
\begin{align*}
&\max\left(f(x_{succ,k}^{M-1})-c_{M-1}\theta^2\delta_{\epsilon}^2,\ldots,f(x_{succ,k}^{1})-c_{1}\theta^2\delta_{\epsilon}^2,f(x_{k})\right) \\
&-\max\left(f(x^{M-2}_{succ,k})-c_{M-1}\theta^2\delta_{\epsilon}^2,\ldots,f(x_{succ,k}^{1})-c_{2}\theta^2\delta_{\epsilon}^2,f(x_{k})-c_{1}\theta^2\delta_{\epsilon}^2,f(x_{k}+\delta_{k}d_{k})\right)\\
&\geq \frac{c}{M}\theta^2\delta_{\epsilon}^2.
\end{align*}
\end{lemma}

\begin{proof}
For ease of notation, we denote that 
$$B_{k}=\max\left(f(x_{succ,k}^{M-1})-c_{M-1}\theta^2\delta_{\epsilon}^2,\ldots,f(x_{succ,k}^{1})-c_{1}\theta^2\delta_{\epsilon}^2,f(x_{k})\right).$$ 
Since $c_{i}$, $1\leq i\leq M-1$, are monotonically increasing, one has
\begin{align*}
B_{k}&=\max\left(f(x_{succ,k}^{M-1})-c_{M-1}\theta^2\delta_{\epsilon}^2,\ldots,f(x_{succ,k}^{1})-c_{1}\theta^2\delta_{\epsilon}^2,f(x_{k})\right)\\
&\geq\max\left(f(x_{succ,k}^{M-1})-c_{M-1}\theta^2\delta_{\epsilon}^2,\ldots,f(x_{succ,k}^{1})-c_{M-1}\theta^2\delta_{\epsilon}^2,f(x_{k})-c_{M-1}\theta^2\delta_{\epsilon}^2\right)\\
&=\max\left(f(x_{succ,k}^{M-1}),\ldots,f(x_{succ,k}^{1}),f(x_{k})\right)-c_{M-1}\theta^2\delta_{\epsilon}^2,
\end{align*}
which, together with~(\ref{sdcdpss}) and Lemma~\ref{lemma2.2}, implies
\begin{align}
B_{k}-f(x_{k}+\delta_{k}d_{k})&\geq\max\left(f(x_{succ,k}^{M-1}),\ldots,f(x_{succ,k}^{1}),f(x_{k})\right)-c_{M-1}\theta^2\delta_{\epsilon}^2-f(x_{k}+\delta_{k}d_{k})\notag\\
&\geq c\delta_{k}^2-c_{M-1}\theta^2\delta_{\epsilon}^2\notag\\
&\geq\frac{c}{M}\theta^2\delta_{\epsilon}^2.\label{22}
\end{align}
Note that
\begin{align*}
&B_{k}-\max\left(f(x^{M-2}_{succ,k})-c_{M-1}\theta^2\delta_{\epsilon}^2,\ldots,f(x_{succ,k}^{1})-c_{2}\theta^2\delta_{\epsilon}^2,f(x_{k})-c_{1}\theta^2\delta_{\epsilon}^2,f(x_{k}+\delta_{k}d_{k})\right)\\
&=B_{k}+\min\left(c_{M-1}\theta^2\delta_{\epsilon}^2-f(x^{M-2}_{succ,k}),\ldots,c_{1}\theta^2\delta_{\epsilon}^2-f(x_{k}),-f(x_{k}+\delta_{k}d_{k})\right)\\
&=\min\left(B_{k}+c_{M-1}\theta^2\delta_{\epsilon}^2-f(x^{M-2}_{succ,k}),\ldots,B_{k}+c_{1}\theta^2\delta_{\epsilon}^2-f(x_{k}),B_{k}-f(x_{k}+\delta_{k}d_{k})\right).
\end{align*}
Using the definition of $B_{k}$ and~(\ref{22}) gives
\begin{align*}
&\min\left(B_{k}+c_{M-1}\theta^2\delta_{\epsilon}^2-f(x^{M-2}_{succ,k}),\ldots,B_{k}+c_{1}\theta^2\delta_{\epsilon}^2-f(x_{k}),B_{k}-f(x_{k}+\delta_{k}d_{k})\right)\\
&\geq\min\left(c_{M-1}\theta^2\delta_{\epsilon}^2-c_{M-2}\theta^2\delta_{\epsilon}^2,\ldots,c_{2}\theta^2\delta_{\epsilon}^2-c_{1}\theta^2\delta_{\epsilon}^2,c_{1}\theta^2\delta_{\epsilon}^2,\frac{c}{M}\theta^2\delta_{\epsilon}^2\right)\\
&=\frac{c}{M}\theta^2\delta_{\epsilon}^2,
\end{align*}
which concludes the proof.
\end{proof}

As the merit function decrease a fixed amount after each successful iteration, the following theorem bounds the number of successful iterations before the function value is reduced to optimality. As a result, we obtain a bound on the total number of iterations for Algorithm~\ref{algo1pss}.
\begin{theorem}
Let Assumption~\ref{a21} and~\ref{a22} hold. Let $c_{i}=\frac{i}{M}c$, $1\leq i\leq M-1$. In Algorithm~\ref{algo1pss}, suppose that $\delta_{\epsilon}\leq\delta_{0}$. We have
\begin{align*}
N_{succ,T_{\epsilon}}\leq\frac{f(x_{0})-f^{*}}{\frac{c}{M}\theta^2\delta_{\epsilon}^2}.
\end{align*}
Furthermore,
\begin{align*}
T_{\epsilon}\leq\frac{\ln{\frac{\delta_{0}\sqrt{n}(\Lf+2c)}{2\theta\epsilon}}}{\ln{\frac{1}{\theta}}}+\frac{\ln{\frac{\gamma}{\theta}}}{\ln{\frac{1}{\theta}}}\frac{nM(\Lf+2c)^2}{4c\theta^2\epsilon^2}\left(f\left(x_{0}\right)-f^{*}\right).
\end{align*}
\end{theorem}

\begin{proof}
Let $s_{k}$ denote the index of the $k$-th successful iteration in Algorithm~\ref{algo1pss}. 
Lemma~\ref{lemma2.4} and the mechanism of Algorithm~\ref{algo1pss} give
\begin{align*}
\sum_{0\leq k\leq T_{\epsilon}}\left(\phi_{k}-\phi_{k+1}\right)&=\sum_{\{j:0\leq s_{j}\leq T_{\epsilon}\}}\left(\phi_{s_{j}}-\phi_{s_{j}+1}\right)\\
&\geq\sum_{\{k:0\leq s_{k}\leq T_{\epsilon}\}}\frac{c}{M}\theta^2\delta_{\epsilon}^2\\
&\geq\sum_{\{k:0\leq s_{k}\leq T_{\epsilon}-1\}}\frac{c}{M}\theta^2\delta_{\epsilon}^2\\
&=N_{succ,T_{\epsilon}}\frac{c}{M}\theta^2\delta_{\epsilon}^2.
\end{align*}
We note that $\sum_{0\leq k\leq T_{\epsilon}}\left(\phi_{k}-\phi_{k+1}\right)=\phi_{0}-\phi_{T_{\epsilon}+1}\leq\phi_{0}$. Therefore, we have
\begin{align*}
N_{succ,T_{\epsilon}}\frac{c}{M}\theta^2\delta_{\epsilon}^2\leq\phi_{0}.
\end{align*}
Note that, by definition, $x_{succ,0}^{0}=\cdots=x_{succ,0}^{M-1}=x_{0}$, which implies $\phi_{0}=f(x_{0})-f^{*}$. Then, applying Lemma~\ref{lemma2.3} with $k=T_{\epsilon}$ delivers the claimed result after substituting $\delta_{\epsilon}=\frac{2\epsilon}{\sqrt{n}(\Lf+2c)}$.
\end{proof}

The complexity of Algorithm~\ref{algo1pss} is $\mathcal{O}(n\epsilon^{-2})$, which matches the complexity of standard direct-search methods in number of iterations.

\subsection{Max-$M$ non-monotone direct search with probabilistic descent}

In this section, we present Algorithm~\ref{algo1}, which replaces the complete positive-spanning-set poll by a single random direction. This practice was first used in the probabilistic-descent assumption~\cite{gratton2015direct}, reducing the search cost per iteration at the expense of guaranteeing a useful direction only with a sufficiently large probability. In contrast to the monotone method in~\cite{gratton2015direct},  Algorithm~\ref{algo1} uses the maximum of the last $M$ distinct iterates in the decrease condition and chooses one random direction $D_k$ in each step. We note that a practical version of Algorithm~\ref{algo1} searches for a direction from $\{D_k, - D_k\}$ where $D_k$ is sampled uniformly at random from the unit sphere. 

\begin{algorithm}[H]
\caption{Max-$M$ Non-monotone Direct Search Based on Probabilistic Descent}\label{algo1}
\begin{algorithmic}[1]
\STATE{Initialization. Choose integer $M>1$, $c>0$, $X_0$, $\Delta_0$, $\theta\in(0,1)$, $\gamma\in(1,\infty)$.}
\FOR{$k=0,1,\ldots$}
\STATE{Uniformly select a random direction $D_k$ from the unit sphere.}
\STATE{For $i\in\mathbb{N}^{+}$, if there is an $i$-th latest successful iterate point from $X_{k}$, then record it with~$X_{succ,k}^{i}$. Otherwise let $X_{succ,k}^{i}$ be $X_{0}$.}
\IF{
\begin{align}\label{sdcd}
\max\left(f(X_{succ,k}^{M-1}),\ldots,f(X_{succ,k}^{1}),f(X_{k})\right)-f(X_{k}+\Delta_{k}D_k)\geq c\Delta_k^2,
\end{align}
}
\STATE{Set $X_{k+1}=X_k+\Delta_k D_k$ and $\Delta_{k+1}=\gamma\Delta_{k}$. This iteration is successful.}
\ELSE
\STATE{Set $X_{k+1}=X_k$ and $\Delta_{k+1}=\theta\Delta_{k}$. This iteration is not successful.}
\ENDIF
\ENDFOR
\end{algorithmic}
\end{algorithm}

For the purpose of the merit-function analysis, we record the successful stepsizes $\Delta_{succ,k}^{i}$  associated with the $i$-th latest successful iterate $X_{succ,k}^{i}$ for $i\in\mathbb{N}^{+}$. Hence $\Delta_{succ,k}^{i}=\Delta_{0}$ for $X_{succ,k}^{i}=X_{0}$. We also use the notations
$X_{succ,k}^{0}=X_{k}$ and $\Delta_{succ,k}^{0}=\Delta_{k}$ for convenience. 

To formalize conditioning on the past, let $\mathcal{F}_{k}$ denote the $\sigma$-algebra generated by randomness up to the end of iteration $k-1$. For Algorithm~\ref{algo1}, the randomness is from the previous polling directions $D_0,\ldots,D_{k-1}$. Similar to \cite{gratton2015direct}, we assume the following probabilistic-descent condition on the random directions.
\begin{assumption}\label{a11}
The sequence $D_k$ is $p$-probabilistically $\kappa$-descent, i.e., 
there exist $\kappa\in(0,1]$ and $p\in(0,1]$ such that, for every $k<T_\epsilon$,
\begin{align*}
P\left(\text{cm}(D_k, -\nabla f(X_k))\geq \kappa\,\middle|\,\mathcal{F}_k\right)\geq p, 
\end{align*}
where $\text{cm}(a, b)$ denotes the cosine measure 
for the vectors $a, b \in \mathbb{R}^n$. 
\end{assumption}
Assumption~\ref{a11} states that the cosine measure of the random direction and the descent direction is favorable for some probability $p$. Hence, it is verifiable and independent of the algorithmic construction. For example, if $D_k$ is uniformly generated from a unit sphere, the direction is descent with $\kappa \geq 1/ (7\sqrt{n})$ for a sufficiently large probability \cite{gratton2015direct}. This assumption enables the acceptance of the non-monotone decrease condition, which we demonstrate below. 

\begin{lemma}\label{lem:prob-success}
Let Assumptions~\ref{a22} and~\ref{a11} hold and define
\begin{align*}
\delta_\epsilon=\frac{2\kappa\epsilon}{\Lf+2c}.
\end{align*}
Then, for every $k<T_\epsilon$,
\begin{align*}
P\left(\max\left(f(X_{succ,k}^{M-1}),\ldots,f(X_{succ,k}^{1}),f(X_{k})\right)-f(X_{k}+\Delta_{k}D_k)\geq c\Delta_k^2 \mid\mathcal F_k\right)\mathbf 1_{\{\Delta_k\leq\delta_\epsilon\}}
\geq p\mathbf 1_{\{\Delta_k\leq\delta_\epsilon\}}.
\end{align*}
\end{lemma}

\begin{proof}
Using the argument in Lemma~\ref{lemma2.1} we have 
\begin{align}
f(X_k)-f(X_k+\Delta_kD_k)&\geq-\nabla f(X_k)^\top (\Delta_kD_k)-\frac{L}{2}\Delta_k^2.
\end{align}
Using the cosine measure in Assumption~\ref{a11}, $\|\nabla f(X_k)\| \geq \epsilon $ and $\|D_k\| = 1$ , we have 
\begin{align*}
f(X_k)-f(X_k+\Delta_kD_k)
&\geq\kappa\epsilon\Delta_k-\frac{\Lf}{2}\Delta_k^2
\geq c\Delta_k^2,
\end{align*}
where the last line follows from the fact that $\Delta_k\leq\delta_\epsilon$. 
Hence, the max-$M$ acceptance condition holds for every $k \leq T_\epsilon$. Taking conditional probabilities proves the result.
\end{proof}

Recall from~(\ref{meritd}) that the non-negative merit function sequence is defined by
\begin{align}\label{merits}
\Phi_k=\max\left(f(X_{succ,k}^{M-1})-c_{M-1}\Delta_{k}^{2},\ldots,f(X_{succ,k}^{1})-c_{1}\Delta_{k}^{2},f(X_{k})\right)-f^*+\eta\Delta_k^2, 
\end{align}
where $0<c_1<\cdots<c_{M-1}<\eta$. 
The following lemma specifies the constants $\{c_i\}$ and bounds the change in the first term of this merit function when the iteration is successful.

\begin{lemma}\label{l11}
Let $c_{i}=\frac{c}{\gamma^2}\frac{1-\gamma^{-2i}}{1-\gamma^{-2M}}$, $1\leq i\leq M-1$. In Algorithm~\ref{algo1}, from the satisfaction of the sufficient decrease condition~(\ref{sdcd}) at iteration $k$, we have
\begin{align*}
&\max\left(f(X_{succ,k}^{M-1})-c_{M-1}\Delta_{k}^{2},\ldots,f(X_{succ,k}^{1})-c_{1}\Delta_{k}^{2},f(X_{k})\right) \\
&-\max\left(f(X^{M-2}_{succ,k})-c_{M-1}\gamma^2\Delta_{k}^{2},\ldots,f(X_{succ,k}^{1})-c_{2}\gamma^2\Delta_{k}^{2},f(X_{k})-c_{1}\gamma^2\Delta_{k}^{2},f(X_{k}+\Delta_{k}D_{k})\right)\\
&\geq c\frac{1-\gamma^{-2}}{1-\gamma^{-2M}}\Delta_k^2.
\end{align*}
\end{lemma}

\begin{proof}
Denote $B_{k}=\max\left(f(X_{succ,k}^{M-1})-c_{M-1}\Delta_{k}^{2},\ldots,f(X_{succ,k}^{1})-c_{1}\Delta_{k}^{2},f(X_{k})\right)$. Since $c_{i}$, $1\leq i\leq M-1$, are monotonically increasing, one has
\begin{align*}
B_{k}&=\max\left(f(X_{succ,k}^{M-1})-c_{M-1}\Delta_{k}^{2},\ldots,f(X_{succ,k}^{1})-c_{1}\Delta_{k}^{2},f(X_{k})\right)\\
&\geq\max\left(f(X_{succ,k}^{M-1})-c_{M-1}\Delta_{k}^{2},\ldots,f(X_{succ,k}^{1})-c_{M-1}\Delta_{k}^{2},f(X_{k})-c_{M-1}\Delta_{k}^{2}\right)\\
&=\max\left(f(X_{succ,k}^{M-1}),\ldots,f(X_{succ,k}^{1}),f(X_{k})\right)-c_{M-1}\Delta_{k}^{2},
\end{align*}
which, together with~(\ref{sdcd}), implies
\begin{align}
B_{k}-f(X_{k}+\Delta_{k}D_{k})&\geq\max\left(f(X_{succ,k}^{M-1}),\ldots,f(X_{succ,k}^{1}),f(X_{k})\right)-c_{M-1}\Delta_{k}^{2}-f(X_{k}+\Delta_{k}D_{k})\notag\\
&\geq (c-c_{M-1})\Delta_{k}^2.\label{l111}
\end{align}
Note that
\begin{align*}
&B_{k}-\max\left(f(X^{M-2}_{succ,k})-c_{M-1}\gamma^2\Delta_{k}^{2},\ldots,f(X_{succ,k}^{1})-c_{2}\gamma^2\Delta_{k}^{2},f(X_{k})-c_{1}\gamma^2\Delta_{k}^{2},f(X_{k}+\Delta_{k}D_{k})\right)\\
&=B_{k}+\min\left(c_{M-1}\gamma^2\Delta_{k}^{2}-f(X^{M-2}_{succ,k}),\ldots,c_{1}\gamma^2\Delta_{k}^{2}-f(X_{k}),-f(X_{k}+\Delta_{k}D_{k})\right)\\
&=\min\left(B_{k}+c_{M-1}\gamma^2\Delta_{k}^{2}-f(X^{M-2}_{succ,k}),\ldots,B_{k}+c_{1}\gamma^2\Delta_{k}^{2}-f(X_{k}),B_{k}-f(X_{k}+\Delta_{k}D_{k})\right).
\end{align*}
Using the definition of $B_{k}$ and~(\ref{l111}) gives
\begin{align*}
&\min\left(B_{k}+c_{M-1}\gamma^2\Delta_{k}^{2}-f(X^{M-2}_{succ,k}),\ldots,B_{k}+c_{1}\gamma^2\Delta_{k}^{2}-f(X_{k}),B_{k}-f(X_{k}+\Delta_{k}D_{k})\right)\notag\\
&\geq\min\left(c_{M-1}\gamma^2\Delta_k^2-c_{M-2}\Delta_k^2,\ldots,c_{2}\gamma^2\Delta_k^2-c_{1}\Delta_k^2,c_{1}\gamma^2\Delta_k^2,(c-c_{M-1})\Delta_{k}^2\right),
\end{align*}
which concludes the proof after some basic calculations.
\end{proof}

We  present the sufficient decrease lemma for the merit function of Algorithm~\ref{algo1} below.
\begin{lemma}\label{l3.2}
There exist constants $c_{i}=\frac{c}{\gamma^2}\frac{1-\gamma^{-2i}}{1-\gamma^{-2M}}$, $1\leq i\leq M-1$, $\eta=\frac{c\left(\gamma^2-\theta^2-(1-\theta^2)\gamma^{2-2M}\right)}{\gamma^2(1-\gamma^{-2M})(\gamma^2-\theta^2)}$, and $\nu=\frac{c(\gamma^2-1)(1-\theta^2)\gamma^{-2M}}{(1-\gamma^{-2M})(\gamma^2-\theta^2)}>0$ such that
\begin{align*}
    \Phi_k-\Phi_{k+1}\geq\nu\Delta^2_k.
\end{align*}
\end{lemma}

\begin{proof}
We consider separately whether iteration $k$ is successful or not. Using the mechanism of the algorithm and Lemma~\ref{l11}, if iteration $k$ is successful, then~(\ref{sdcd}) holds and
\begin{align*}
&\Phi_k-\Phi_{k+1}\\
&=\max\left(f(X_{succ,k}^{M-1})-c_{M-1}\Delta_{k}^{2},\ldots,f(X_{succ,k}^{1})-c_{1}\Delta_{k}^{2},f(X_{k})\right)+\eta(1-\gamma^2)\Delta_{k}^{2}\\
&-\max\left(f(X^{M-2}_{succ,k})-c_{M-1}\gamma^2\Delta_{k}^{2},\ldots,f(X_{succ,k}^{1})-c_{2}\gamma^2\Delta_{k}^{2},f(X_{k})-c_{1}\gamma^2\Delta_{k}^{2},f(X_{k}+\Delta_{k}D_{k})\right)\\
&\geq\left(c\frac{1-\gamma^{-2}}{1-\gamma^{-2M}}+\eta(1-\gamma^2)\right)\Delta_k^2\\
&=\nu\Delta_k^2.
\end{align*}
We will use the facts that $\max(a,b,c)=\max(c,\max(a,b))$ and that if $x\leq y$, then $\max(a,x)-\max(a,y)\geq x-y$. Also note that $a-\max(b,c)=\min(a-b,a-c)$. These observations, together with the mechanism of the algorithm, imply that, if iteration $k$ is not successful, then
\begin{align*}
\Phi_k-\Phi_{k+1}=&\max\left(f(X_{succ,k}^{M-1})-c_{M-1}\Delta_{k}^{2},\ldots,f(X_{succ,k}^{1})-c_{1}\Delta_{k}^{2},f(X_{k})\right)+\eta(1-\theta^2)\Delta_{k}^{2}\\
&-\max\left(f(X_{succ,k}^{M-1})-c_{M-1}\theta^2\Delta_{k}^{2},\ldots,f(X_{succ,k}^{1})-c_{1}\theta^2\Delta_{k}^{2},f(X_{k})\right)\\
\geq&\max\left(f(X_{succ,k}^{M-1})-c_{M-1}\Delta_{k}^{2},\ldots,f(X_{succ,k}^{1})-c_{1}\Delta_{k}^{2})\right)+\eta(1-\theta^2)\Delta_k^2\\
&-\max\left(f(X_{succ,k}^{M-1})-c_{M-1}\theta^2\Delta_{k}^{2},\ldots,f(X_{succ,k}^{1})-c_{1}\theta^2\Delta_{k}^{2}\right)\\
\geq&\min\left((\theta^2-1)c_{M-1},\ldots,(\theta^2-1)c_{1}\right)\Delta_k^2 +\eta(1-\theta^2)\Delta_k^2\\
=&(\theta^2-1)c_{M-1}\Delta_k^2 +\eta(1-\theta^2)\Delta_k^2\\
=&\nu\Delta_k^2.
\end{align*}
Therefore, from the algorithm, one has
\begin{align*}
\Phi_k-\Phi_{k+1}&=(\Phi_k-\Phi_{k+1})(\mathbf{1}\{\text{iteration $k$ is successful}\}+\mathbf{1}\{\text{iteration $k$ is not successful}\})\\
&\geq\nu\Delta_k^2(\mathbf{1}\{\text{iteration $k$ is successful}\}+\mathbf{1}\{\text{iteration $k$ is not successful}\})\\
&=\nu\Delta_k^2.
\end{align*}
\end{proof}

We now explain how the stopping-time framework of Blanchet et al.~\cite{blanchet2019convergence} applies. This framework relies on the two following conditions.
\begin{enumerate}
    \item There exists a fixed threshold $\delta_\epsilon$ such that for $\Delta_k \leq \delta_\epsilon$ and $k<T_\epsilon$, the stepsize is multiplied by $\gamma$ with conditional probability at least $p$ and otherwise is multiplied by $\theta$. This is proved in Lemma~\ref{lem:prob-success}. In addition, we assume that $p\log\gamma+(1-p)\log\theta > 0$
    , which prevents the method from spending too many iterations at stepsizes much smaller than $\delta_\epsilon$.
    \item For $k<T_\epsilon$, the merit function satisfies $E[\Phi_k-\Phi_{k+1}|\mathcal{F}_{k}]\geq\nu\Delta^2_k$ for some values $\nu > 0$. This is proved in Lemma~\ref{l3.2}. 
\end{enumerate}
With these two conditions, the framework in \cite{blanchet2019convergence,ding2025sequential} studies a renewal–reward martingale process generated by the stepsize $\Delta_k$ and shows an upper bound for the expected stopping time $E[T_\epsilon]$ of the algorithm. We note that Blanchet et al.~\cite{blanchet2019convergence} analyzed the standard case with $\gamma \theta = 1$ and $p > \frac{1}{2}$, while Ding et al.~\cite{ding2025sequential} subsequently stated the corresponding bound for a general success probability and multiplicative factors that satisfy $p\log\gamma+(1-p)\log\theta > 0$.

\begin{theorem}[expected number of iterations]\label{t3.9}
Let Assumptions~\ref{a21},~\ref{a22}, and~\ref{a11} hold, let
$\delta_\epsilon=2\kappa\epsilon/(\Lf+2c)$, and suppose that $p\ln\gamma+(1-p)\ln\theta>0$.
Then
\begin{align*}
    E[T_\epsilon-1]\leq \frac{\ln{\gamma}}{p\ln{\gamma}+(1-p)\ln{\theta}}\cdot\frac{\Phi_0}{\nu\theta^2\delta_{\epsilon}^2} = \frac{\ln{\gamma}}{p\ln{\gamma}+(1-p)\ln{\theta}}\cdot\frac{\Phi_0 (\Lf+2c)^2}{4\nu\theta^2 \kappa^2\epsilon^2},
\end{align*}
where the first bound follows from the stopping-time framework~\cite{blanchet2019convergence}.
Consequently, the expected iteration complexity of Algorithm~\ref{algo1} is $\mathcal{O}(\epsilon^{-2})$.
\end{theorem}

The explicit constant $\nu$ in Lemma~\ref{l3.2} decreases geometrically with~$M$, so the resulting bound contains a factor of order $\gamma^{2M}$. In contrast, the complete-poll analysis gives linear dependence on~$M$. The geometric factor is a consequence of the particular feasible correction coefficients used here, not a lower bound for every max-$M$ method; hence, the possibility of sharper coefficients is open.

We note that another merit function $\Psi_{k}$ for Algorithm~\ref{algo1} can be constructed differently from~$\Phi_{k}$, which we present in detail in Appendix~\ref{app:merit}.
When extending to the stochastic objective functions, $\Psi_{k}$ has a disadvantage. 
Specifically, the decrease of $\Psi_{k}$ relies on the exactness of~$M$ sufficient decrease condition evaluations; therefore, error control in the algorithm for the stochastic objective function becomes relatively difficult when the stepsize is adaptive. The benefit of~$\Phi_{k}$ is that its decrease depends merely on the current sufficient decrease condition evaluation and extends more naturally to the stochastic case, as shown in the next section.

\section{Non-monotone direct search for stochastic functions}

A challenge in the stochastic setting arises when the past and the current stochastic function estimates in the decrease condition are not independent. A common technique to overcome this challenge is resampling the random estimates in new iterations. We propose a stochastic max-$M$ non-monotone direct-search algorithm, following this principle. 
Let $f(x)$ be the objective function of our interest and let $F(x,\xi)$ be its stochastic oracle. 
Our method extends Algorithm~\ref{algo1} by replacing the exact objective value in the non-monotone condition with an independent stochastic estimate, which we presented in Algorithm~\ref{algo2} as follows.

\begin{algorithm}[H]
\caption{Stochastic Max-$M$ Non-monotone Direct Search}\label{algo2}
\begin{algorithmic}[1]
\STATE{Initialization. Choose integer $M>1$, $c>0$, $X_0$, $\Delta_0$, $\theta\in(0,1)$, $\gamma\in(1,\infty)$.}
\FOR{$k=0,1,\ldots$}
\STATE{Uniformly select a random direction $D_k$ from the unit sphere.}
\STATE{For $i\in\mathbb{N}^{+}$, if there is an $i$-th latest successful iterate point from $X_{k}$, then record it with~$X_{succ,k}^{i}$. Otherwise let $X_{succ,k}^{i}$ be $X_{0}$. Let $X^{d}_{k}=X_{k}+\Delta_{k}D_{k}$ be the candidate point.}
\STATE{Query the stochastic oracle at $X_{succ,k}^{i}$, $1\leq i\leq M-1$, $X_{k}$, and $X^{d}_{k}$. Denote the query results by $F(X_{succ,k}^{i},\xi^{i}_{k,1})$, $1\leq i\leq M-1$, $F(X_{k},\xi_{k,2})$, and $F(X^{d}_{k},\xi_{k,3})$. Here the random variables $\xi^{i}_{k,1}$, $1\leq i\leq M-1$, $\xi_{k,2}$, and $\xi_{k,3}$ are independent.}
\IF{$\max\left(F(X_{succ,k}^{M-1},\xi^{M-1}_{k,1}),\ldots,F(X_{succ,k}^{1},\xi^{1}_{k,1}),F(X_{k},\xi_{k,2})\right)-F(X^{d}_{k},\xi_{k,3})\geq c\Delta_{k}^2$,}
\STATE{Set $X_{k+1}=X^{d}_{k}$ and $\Delta_{k+1}=\gamma\Delta_{k}$. This iteration is successful.}
\ELSE
\STATE{Set $X_{k+1}=X_k$ and $\Delta_{k+1}=\theta\Delta_{k}$. This iteration is not successful.}
\ENDIF
\ENDFOR
\end{algorithmic}
\end{algorithm}
\allowdisplaybreaks
We denote the following quantities for the analysis of Algorithm~\ref{algo2}.
\begin{align*}
    G_{k}&=\max\left(F(X_{succ,k}^{M-1},\xi^{M-1}_{k,1}),\ldots,F(X_{succ,k}^{1},\xi^{1}_{k,1}),F(X_{k},\xi_{k,2})\right)-F(X^{d}_{k},\xi_{k,3}),\\
    H_{k}&=\max\left(f(X_{succ,k}^{M-1}),\ldots,f(X_{succ,k}^{1}),f(X_{k})\right)-f(X^{d}_{k}),\\
    \varepsilon_{k}&=G_{k}-H_{k}.
\end{align*}
Here, $\varepsilon_{k}$ is the sampling noise of $G_{k}$ when estimating $H_{k}$. Hence, the random sources in the algorithm are $\{D_{k}\}$ and $\{\varepsilon_{k}\}$. Similar to the previous section, we let $\mathcal{F}_{k}$ be the $\sigma$-algebra generated by randomness up to the end of iteration $k-1$. In addition, let $\mathcal{F}_{k+1/2}$ denote the $\sigma$-algebra generated by randomness up to $D_{k}$. Therefore, $X^d_k$ and $H_{k}$ are $\mathcal{F}_{k+1/2}$-measurable; $G_{k}$, $X_{k+1}$, and $\Delta_{k+1}$ are $\mathcal{F}_{k+1}$-measurable. We denote the event $A_{k}:=\mathbf{1}(\text{iteration $k$ is successful})$.
Similar to the deterministic setting, we aim to bound $T_{\epsilon}$ for a given threshold $\epsilon>0$.
We impose the following assumption on the sampling error $\epsilon_k$.

\begin{assumption}\label{a23}
The stochastic oracle is accurate such that the distribution of $\varepsilon_{k}$ satisfies for each $k$
\begin{align*}
P(\varepsilon_{k}\geq0|\mathcal{F}_{k+1/2})&\geq p_{0}\\
P(\varepsilon_{k}\geq b|\mathcal{F}_{k+1/2})&\leq \frac{\beta\Delta_{k}^2}{b},\quad\forall b>0.
\end{align*}
\end{assumption}
{Assumption~\ref{a23} controls the two ways in which sampling error can affect the acceptance decision. The first condition ensures with probability at least $p_0$ that the estimated decrease does not understate the true decrease, while the second condition bounds the upper tail of the error. We note that this assumption concerns the error bound between $H_k$ and its stochastic estimate $G_k$, not for each function observation taken individually. The first condition can be attained for $p_0 = 1/2$ if the individual errors are symmetric. In the case of non-symmetric errors, the value of $p_0$ depends on the distribution of those errors. In general, our analysis does not require a lower bound on $p_0$ and thus the first condition is very mild. The second condition is a one-sided tail bound, which is common in the stochastic direct search literature \cite{rinaldi2024stochastic}. In Appendix~\ref{app:sampling}, we show that this condition can be satisfied when the number of samples are $\mathcal{O}(\Delta_k^{-4})$. We note that this order of samples is also standard in the literature \cite{rinaldi2024stochastic} to ensure the same (deterministic) iteration complexity for stochastic direct search.} 

We also need Assumption~\ref{a11}, which assumes good behavior of the random direction and therefore is independent of the stochastic oracles. By Lemma~\ref{lem:prob-success}, 
for every $k<T_\epsilon$, we have
\begin{align}\label{eq:hk}
P\left(H_k\geq c\Delta_k^2 \mid\mathcal F_k\right)\mathbf 1_{\{\Delta_k\leq\delta_\epsilon\}}
\geq p\mathbf 1_{\{\Delta_k\leq\delta_\epsilon\}}.
\end{align}
Independence between the direction and the subsequent oracle samples allows this probability and the probability in Assumption~\ref{a23} to combine into $pp_0$ in the following lemma.

\begin{lemma}\label{l4.1}
Let Assumptions~\ref{a22},~\ref{a11}, and~\ref{a23} hold and $\delta_\epsilon=\frac{2\kappa\epsilon}{\Lf+2c}$.
For any $\epsilon>0$ and $k<T_\epsilon$, we obtain
\begin{align*}
P(A_k=1|\mathcal{F}_{k})\mathbf{1}_{\{\Delta_k\leq\delta_\epsilon\}}\geq pp_{0}\mathbf{1}_{\{\Delta_k\leq\delta_\epsilon\}}.
\end{align*}
\end{lemma}

\begin{proof}
Note that $\{H_{k}\geq c\Delta_{k}^2\}\cap\{G_{k}-H_{k}\geq0\}\subseteq\{A_{k}=1\}$. Therefore, we have
\begin{align*}
P(A_k=1|&\mathcal{F}_{k})\geq P\left(\{H_{k}\geq c\Delta_{k}^2\}\cap\{G_{k}-H_{k}\geq0\}|\mathcal{F}_{k}\right)
\end{align*}
Using the tower property of conditional expectation and the definition of $\varepsilon$, from Assumption~\ref{a23}, we have
\begin{align*}
P(A_k=1|\mathcal{F}_{k})\geq&E[\mathbf{1}_{\{H_{k}\geq c\Delta_{k}^2\}}P\left(G_{k}-H_{k}\geq0|\mathcal{F}_{k+1/2}\right)|\mathcal{F}_{k}]\\
=&E[\mathbf{1}_{\{H_{k}\geq c\Delta_{k}^2\}}P\left(\varepsilon_{k}\geq0|\mathcal{F}_{k+1/2}\right)|\mathcal{F}_{k}]\\
\geq&E[\mathbf{1}_{\{H_{k}\geq c\Delta_{k}^2\}}p_{0}|\mathcal{F}_{k}]\\
=&p_{0}P\left(H_{k}\geq c\Delta_{k}^2|\mathcal{F}_{k}\right).
\end{align*}
Together with the bound~\eqref{eq:hk}, we have
\begin{align*}
P(A_k=1|\mathcal{F}_{k})\mathbf{1}_{\{\Delta_k\leq\delta_\epsilon\}}\geq&p_{0}P\left(H_{k}\geq c\Delta_{k}^2|\mathcal{F}_{k}\right)\mathbf{1}_{\{\Delta_k\leq\delta_\epsilon\}}\\
\geq&pp_{0}\mathbf{1}_{\{\Delta_k\leq\delta_\epsilon\}}.
\end{align*}

\end{proof}

The merit function $\Phi_k$ is defined similarly as in \eqref{merits}. The next Lemma establishes the coefficients and the expected decrease of $\Phi_k$ as follows.
\begin{lemma}\label{l4.2}
Let Assumption~\ref{a23} hold for  $\beta=\frac{1}{2}(\eta-c_{M-1})(1-\theta^2) > 0$. There exist $c_{M-1}>0$, $c_{i}=\frac{1-\gamma^{-2i}}{1-\gamma^{2-2M}}c_{M-1}$, $1\leq i\leq M-2$, $\eta>c_{M-1}$, and $\nu>0$ such that
\begin{align*}
E[\Phi_k-\Phi_{k+1}|\mathcal{F}_{k+1/2}]\geq\nu\Delta^2_k.
\end{align*}
\end{lemma}

\begin{proof}
Selecting $c_{M-1}>0$, $\eta>c_{M-1}$, and $\nu>0$ such that
\begin{align*}
\nu&=\frac{\gamma^2-1}{1-\gamma^{2-2M}}c_{M-1}+\eta(1-\gamma^2)=\frac{1}{2}(\eta-c_{M-1})(1-\theta^2)
\end{align*}
and
\begin{align}
c=c_{M-1}\theta^{2}+\eta\gamma^2-\eta\theta^2.\label{c}
\end{align}
This gives
\begin{align*}
c_{M-1}&=c\frac{(2\gamma^2-\theta^2-1)(1-\gamma^{2-2M})
}{\text{aux}}\\
\eta&=c\frac{
2\gamma^2-2+(1-\theta^2)(1-\gamma^{2-2M})
}{\text{aux}}\\
\nu&=c\frac{
(1-\theta^2)(\gamma^2-1)\gamma^{2-2M}
}{\text{aux}},
\end{align*}
where $\text{aux}=\theta^2(2\gamma^2-\theta^2-1)(1-\gamma^{2-2M})
+(\gamma^2-\theta^2)\left(2\gamma^2-2+(1-\theta^2)(1-\gamma^{2-2M})\right)$ is positive.

We will use similar arguments as in Lemma~\ref{l3.2}. In particular, we note again that if $x\leq y$, then $\max(a,x)-\max(a,y)\geq x-y$ and that $a-\max(b,c)=\min(a-b,a-c)$. Therefore, if iteration $k$ is not successful, then, from the mechanism of the algorithm,
\begin{align*}
\Phi_k-\Phi_{k+1}=&\max\left(f(X_{succ,k}^{M-1})-c_{M-1}\Delta_{k}^{2},\ldots,f(X_{succ,k}^{1})-c_{1}\Delta_{k}^{2},f(X_{k})\right)+\eta\Delta^2_{k}\\
&-\max\left(f(X_{succ,k}^{M-1})-c_{M-1}\theta^2\Delta_{k}^{2},\ldots,f(X_{succ,k}^{1})-c_{1}\theta^2\Delta_{k}^{2},f(X_{k})\right)-\eta\theta^2\Delta^2_{k}\\
\geq&\max\left(f(X_{succ,k}^{M-1})-c_{M-1}\Delta_{k}^{2},\ldots,f(X_{succ,k}^{1})-c_{1}\Delta_{k}^{2}\right)+\eta\Delta^2_{k}\\
&-\max\left(f(X_{succ,k}^{M-1})-c_{M-1}\theta^2\Delta_{k}^{2},\ldots,f(X_{succ,k}^{1})-c_{1}\theta^2\Delta_{k}^{2}\right)-\eta\theta^2\Delta^2_{k}\\
\geq&\min\left((\theta^2-1)c_{M-1},\ldots,(\theta^2-1)c_{1}\right)\Delta_k^2+\eta\Delta^2_{k}-\eta\theta^2\Delta^2_{k}\\
=&(\eta-c_{M-1})(1-\theta^2)\Delta_k^2.
\end{align*}
\allowdisplaybreaks
Also using the definition of $H_{k}$, the mechanism of the algorithm and $a-\max(b,c)=\min(a-b,a-c)$, if iteration $k$ is successful, then (note also that $c_i < c_{M-1}$)
\begin{align*}
&~\Phi_k-\Phi_{k+1}
\\=&\max\left(f(X_{succ,k}^{M-1})-c_{M-1}\Delta_{k}^{2},\ldots,f(X_{succ,k}^{1})-c_{1}\Delta_{k}^{2},f(X_{k})\right)+\eta\Delta^2_{k}\\
&-\max\left(f(X^{M-2}_{succ,k})-c_{M-1}\gamma^2\Delta_{k}^{2},\ldots,f(X_{k})-c_{1}\gamma^2\Delta_{k}^{2},f(X^{d}_{k})\right)-\eta\gamma^2\Delta^2_{k}\\
\geq&\min\left(c_{M-1}\gamma^2\Delta_k^2-c_{M-2}\Delta_k^2,\ldots,c_{2}\gamma^2\Delta_k^2-c_{1}\Delta_k^2,c_{1}\gamma^2\Delta_k^2,H_{k}-c_{M-1}\Delta_{k}^2\right)+\eta(1-\gamma^2)\Delta_k^2\\
=&\min\left(\frac{\gamma^2-1}{1-\gamma^{2-2M}}c_{M-1}\Delta_{k}^{2},H_{k}-c_{M-1}\Delta_{k}^2\right)+\eta(1-\gamma^2)\Delta_k^2.
\end{align*}

Based on the arguments above, one is able to bound below $\Phi_k-\Phi_{k+1}$, which is $\mathcal{F}_{k+1}$-measurable, by $\Delta_{k}$ and $H_{k}$, which are $\mathcal{F}_{k+1/2}$-measurable, when $A_{k}$ equals $0$ and $1$ separately. Therefore, one has
\begin{align*}
E[\Phi_k-&\Phi_{k+1}|\mathcal{F}_{k+1/2}]=E[(\Phi_k-\Phi_{k+1})\mathbf{1}(A_{k}=1)|\mathcal{F}_{k+1/2}]+E[(\Phi_k-\Phi_{k+1})\mathbf{1}(A_{k}=0)|\mathcal{F}_{k+1/2}]\\
\geq&\left(\min\left(\frac{\gamma^2-1}{1-\gamma^{2-2M}}c_{M-1}\Delta_{k}^{2},H_{k}-c_{M-1}\Delta_k^2\right)+\eta(1-\gamma^2)\Delta_k^2\right)P(A_{k}=1|\mathcal{F}_{k+1/2})+\\
&(\eta-c_{M-1})(1-\theta^2)\Delta_k^2P(A_{k}=0|\mathcal{F}_{k+1/2})\\
=&\left(\min\left(\frac{\gamma^2-1}{1-\gamma^{2-2M}}c_{M-1}\Delta_{k}^{2},H_{k}-c_{M-1}\Delta_k^2\right)+\eta(1-\gamma^2)\Delta_k^2\right)P(A_{k}=1|\mathcal{F}_{k+1/2})+\\
&(\eta-c_{M-1})(1-\theta^2)\Delta_k^2(1-P(A_{k}=1|\mathcal{F}_{k+1/2}))\\
=&\left(\min\left(\frac{\gamma^2-1}{1-\gamma^{2-2M}}c_{M-1}\Delta_{k}^{2}+c_{M-1}\Delta_k^2,H_{k}\right)+\eta(1-\gamma^2)\Delta_k^2-c_{M-1}\Delta_k^2\right)P(A_{k}=1|\mathcal{F}_{k+1/2})+\\
&(\eta-c_{M-1})(1-\theta^2)\Delta_k^2(1-P(A_{k}=1|\mathcal{F}_{k+1/2}))\\
=&\left(\min\left(\frac{\gamma^2-\gamma^{2-2M}}{1-\gamma^{2-2M}}c_{M-1}\Delta_{k}^{2},H_{k}\right)+(\eta\theta^2-\eta\gamma^2-c_{M-1}\theta^2)\Delta_k^2\right)P(A_{k}=1|\mathcal{F}_{k+1/2})+\\
&(\eta-c_{M-1})(1-\theta^2)\Delta_k^2.
\end{align*}

Note that $S^1_k=\{\omega:H_{k}\geq \frac{\gamma^2-\gamma^{2-2M}}{1-\gamma^{2-2M}}c_{M-1}\Delta_{k}^{2}\}$ and $S^2_k=\{\omega:H_{k}<\frac{\gamma^2-\gamma^{2-2M}}{1-\gamma^{2-2M}}c_{M-1}\Delta_{k}^{2}\}$ are two disjoint $\mathcal{F}_{k+1/2}$-measurable events such that $S^1_k\cup S^2_k=\Omega$. It suffices to prove $E[\Phi_k-\Phi_{k+1}|\mathcal{F}_{k+1/2}]\geq\nu\Delta^2_k$ for these two scenarios separately.

\textbf{Case 1.} First, we consider the case $S_k^1=\{\omega:H_{k}\geq \frac{\gamma^2-\gamma^{2-2M}}{1-\gamma^{2-2M}}c_{M-1}\Delta_{k}^{2}\}$. Then one has
\begin{align*}
E[\Phi_k-\Phi_{k+1}|\mathcal{F}_{k+1/2}]\mathbf{1}_{S_k^1}\geq&\left(\frac{\gamma^2-1}{1-\gamma^{2-2M}}c_{M-1}\Delta_{k}^{2}+\eta(1-\gamma^2)\Delta_k^2\right)P(A_{k}=1|\mathcal{F}_{k+1/2})\mathbf{1}_{S_k^1}+\\
&(\eta-c_{M-1})(1-\theta^2)\Delta_k^2P(A_{k}=0|\mathcal{F}_{k+1/2})\mathbf{1}_{S_k^1}\\
\geq&\min\left(\frac{\gamma^2-1}{1-\gamma^{2-2M}}c_{M-1}\Delta_{k}^{2}+\eta(1-\gamma^2)\Delta_k^2,(\eta-c_{M-1})(1-\theta^2)\Delta_k^2\right)\mathbf{1}_{S_k^1}\\
=&\nu\Delta_{k}^2\mathbf{1}_{S_k^1}.
\end{align*}

\textbf{Case 2.} Otherwise, we consider $S_k^2=\{\omega:H_{k}<\frac{\gamma^2-\gamma^{2-2M}}{1-\gamma^{2-2M}}c_{M-1}\Delta_{k}^{2}\}$.Then one has
\begin{align}
E[\Phi_k-\Phi_{k+1}|\mathcal{F}_{k+1/2}]\mathbf{1}_{S_k^2}\geq&\left(H_{k}+(\eta\theta^2-\eta\gamma^2-c_{M-1}\theta^2)\Delta_k^2\right)P(A_{k}=1|\mathcal{F}_{k+1/2})\mathbf{1}_{S_k^2}+\notag\\
&(\eta-c_{M-1})(1-\theta^2)\Delta_k^2\mathbf{1}_{S_k^2}.\label{case2}
\end{align}
Since $0<\theta<1<\gamma$, from the definition of~$c$ in~(\ref{c}), one can check that
\begin{align*}
H_{k}&<\frac{\gamma^2-\gamma^{2-2M}}{1-\gamma^{2-2M}}c_{M-1}\Delta_{k}^{2}<c\Delta_{k}^2.
\end{align*}
Then, from Assumption~\ref{a23}, the choice of $\beta=\frac{1}{2}(\eta-c_{M-1})(1-\theta^2)$, and~(\ref{c}),
\begin{align*}
P(A_{k}=1|\mathcal{F}_{k+1/2})\mathbf{1}_{S_k^2}&=P(G_{k}\geq c\Delta_{k}^2|\mathcal{F}_{k+1/2})\mathbf{1}_{S_k^2}\\
&=P(\varepsilon_{k}\geq c\Delta_{k}^2-H_{k}|\mathcal{F}_{k+1/2})\mathbf{1}_{S_k^2}\\
&\leq\frac{1}{2}\frac{(\eta-c_{M-1})(1-\theta^2)\Delta_k^2}{c\Delta_k^2-H_{k}}\mathbf{1}_{S_k^2}\\
&=-\frac{1}{2}\frac{(\eta-c_{M-1})(1-\theta^2)\Delta_k^2}{H_{k}+(\eta\theta^2-\eta\gamma^2-c_{M-1}\theta^2)\Delta_k^2}\mathbf{1}_{S_k^2},
\end{align*}
which gives, from~(\ref{case2}),
\begin{align*}
E[\Phi_k-\Phi_{k+1}|\mathcal{F}_{k+1/2}]\mathbf{1}_{S_k^2}\; \geq \; \frac{1}{2}(\eta-c_{M-1})(1-\theta^2)\Delta_k^2\mathbf{1}_{S_k^2}
\; = \; \nu\Delta_{k}^2\mathbf{1}_{S_k^2}.
\end{align*}
\end{proof}
Similar to Theorem~\ref{t3.9}, 
we apply the stopping time framework \cite{blanchet2019convergence,ding2025sequential}. In this case, Lemma~\ref{l4.1} shows the success probability $pp_0$ instead of $p$, and Lemma~\ref{l4.2} bounds the expected decrease of the merit function. We have the following Theorem. 

\begin{theorem}[expected number of iterations]
Let Let Assumptions~\ref{a21},~\ref{a22}, and~\ref{a11} hold and 
$\delta_\epsilon=2\kappa\epsilon/(\Lf+2c)$. In addition, let Assumption~\ref{a23} hold for  $\beta=\frac{1}{2}(\eta-c_{M-1})(1-\theta^2) > 0$ and suppose that $pp_0\ln\gamma+(1-pp_0)\ln\theta>0$. 
Then
\begin{align*}
    E[T_\epsilon-1]\leq \frac{\ln{\gamma}}{pp_{0}\ln{\gamma}+(1-pp_{0})\ln{\theta}}\cdot\frac{\Phi_0}{\nu\theta^2\delta_{\epsilon}^2} = \frac{\ln{\gamma}}{pp_{0}\ln{\gamma}+(1-pp_{0})\ln{\theta}}\cdot\frac{\Phi_0(\Lf+2c)^2}{4\nu\theta^2\kappa^2\epsilon^2},
\end{align*}
where the first bound follows from the stopping-time framework~\cite{blanchet2019convergence}.
Consequently, the expected iteration complexity of Algorithm~\ref{algo2} is $\mathcal{O}(\epsilon^{-2})$.
\end{theorem}

\section{Experiment results}

\subsection{Experimental setup}
In this section, we assess the numerical performance of max-$M$ non-monotone direct search on two sets of problems from the CUTEst collection~\cite{NGould_DOrban_PToint_2015}. The first set consists of 38 problems suggested by~\cite{giovannelli2024limitation}, with dimensions in $\{5,10,100\}$. 
These problems exhibit a range of structural properties, including non-linearity, non-convexity, and partial separability. The second set consists of 53 problems selected from~\cite{gratton2020decoupled,giovannelli2024limitation}, with dimensions ranging from 2 to 50. For each problem in this set, negative curvature was detected by running one of the second-order methods reported in~\cite{CPAvelino_JMoguerza_etal_2011}.
The names and corresponding dimensions of such problems are reported in Appendix~\ref{sec:app_exp}. 

We summarize the performance of all methods across the problems using performance profiles~\cite{more2009benchmarking}, which are briefly reviewed below. 
Let $\mathcal{S}$ be the set of solvers, $\mathcal{P}$ be the set of test problems, and $t_{p,s}>0$ the performance measure obtained by solver $s\in\mathcal{S}$ on problem $p\in\mathcal{P}$ (smaller $t_{p,s}$ is better). We define the performance ratio~$r_{p,s}$ and the performance profile~$\rho_s(\alpha)$ as follows 
\[
r_{p,s} ~=~ \frac{t_{p,s}}{\min\{t_{p,s} ~|~ s \in \mathcal{S}\}}, \quad \rho_s(\alpha) ~=~ \frac{1}{|\mathcal{P}|} \text{size}\{p \in \mathcal{P} ~|~ r_{p,s} \le \alpha\}.
\]
If a solver $s$ fails to satisfy the convergence test on a problem $p$, we set $ r_{p,s} = 2 \max_{p, s} r_{p,s}$. We plot~$\rho_s(\alpha)$ as a curve over~$\alpha$. The solver associated with the highest curve is the one with the best performance.  In particular, $\rho_s(1)$ is the fraction of problems for which solver $s$ performs the best, and the value  $\rho_s(\alpha)$ for large $\alpha$ represents the fraction of problems  solved by solver $s$. These quantities provide measures of relative efficiency and robustness, respectively.

For all methods, we set $\gamma = 2$, $\theta = 0.5$, and $c = 1$. We compare different choices of $M$ in  max-$M$ non-monotone direct search where $M$ takes values in $\{2,5,10,20\}$ and impose a budget of 1000 function evaluations per run. For every solver-problem pair, we 
perform $10$ runs using random seeds $0,1,\dots,9$ and report the average objective-value. In all the experiments, we sample a direction $d$ uniformly from the unit sphere. If this direction does not yield a successful point, we proceed with the the opposite direction $-d$ before generating a new random direction. Since $M=2$ provides the strongest overall performance among the tested values, we use this choice when comparing the max-$M$ non-monotone to monotone methods.


\subsection{Deterministic results}


In this section, we present the results for the case where function values are deterministic. The left panel of Figure \ref{fig:ds} displays the performance profiles comparing the max-$M$ non-monotone direct search (Algorithm~\ref{algo1}) for different values of $M \in \{2, 5, 10, 20\}$. Across both problem sets, the max-$2$ variant clearly dominates the others in terms of both efficiency and robustness. While our empirical and theoretical results confirm that the max-$M$ method converges for different values of $M$, its performance deteriorates as the memory size $M$ increases. This observation is not surprising as historical function values are generally high, which leads to more acceptances and more steps to find good function decrease. 

The right panel of Figure \ref{fig:ds} compares the max-$2$ non-monotone method against standard monotone direct search. For Set 1, the monotone method demonstrates better efficiency, while the max-$2$ non-monotone algorithm is more robust. For Set 2, the max-$2$ non-monotone method shows competitive efficiency and achieves a better robustness compared to the monotone approach. Hence, non-monotone method yields good performance overall, especially when the problem has negative curvature as in Set 2. 

\begin{figure}[H]
    \centering
\includegraphics[width=0.48\linewidth]{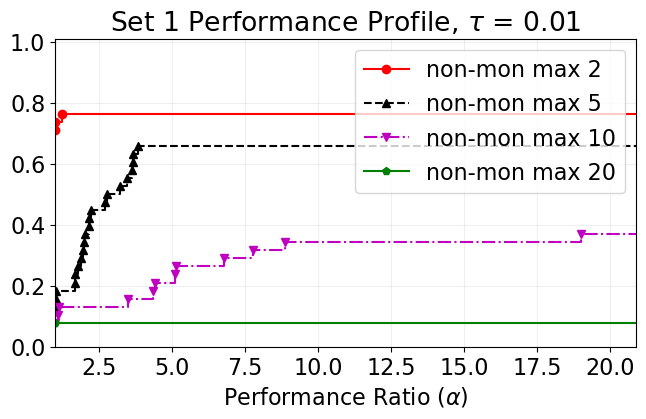}
\includegraphics[width=0.48\linewidth]{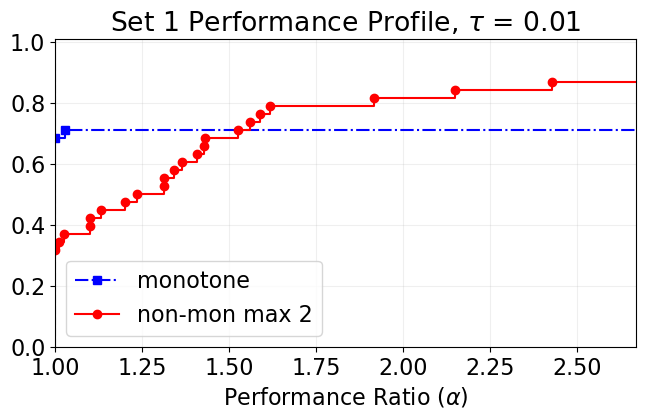}
\includegraphics[width=0.48\linewidth]{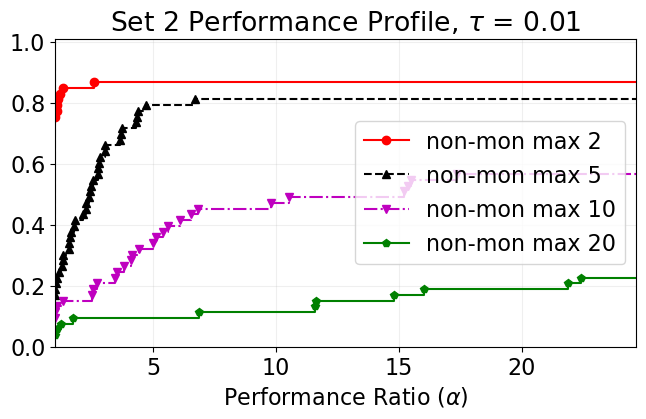}
\includegraphics[width=0.48\linewidth]{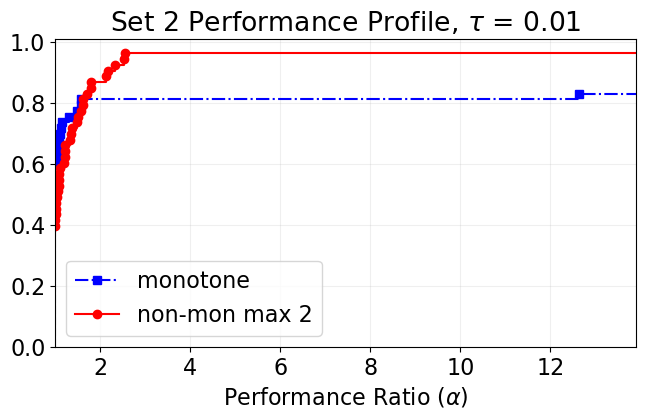}
    \caption{Performance profiles for two problem sets in the deterministic case.}
    \label{fig:ds}
\end{figure}
\subsection{Stochastic results}

For the stochastic experiments, we perturb the objective values with Gaussian noise and choose the 
number of samples of order $\mathcal{O} (\delta_k^4)$. 
We implement Algorithm~\ref{algo2} where the stochastic function values at prior successful iterates are resampled independently. As a result, in this theoretical scenario, the max-$M$ method requires $M+1$ stochastic function estimates per iteration, which is higher than the two estimates for the monotone method. 
Alternatively, we implement a practical scenario, where only the stochastic function value at the candidate point $X_k^d$ is resampled in each iteration. Hence, the per-iteration costs of max-$M$ non-monotone and monotone methods are the same (one stochastic evaluation), while the random estimates used in the acceptance test are no longer independent.

In the theoretical scenario (Figure \ref{fig:sdsr}), where all values are resampled, the monotone direct search outperforms the max-$2$ non-monotone one on Set 1, while on Set 2, max-$2$ non-monotone is better. This is consistent with the observation in the deterministic case that max-$2$ non-monotone offers significant improvement when the problems have negative curvature. 
In the practical scenario (Figure \ref{fig:sds}), where only the new point is sampled and the per-iteration cost is the same for both methods, the performance of max-$2$ non-monotone algorithm is better than the monotone one in both initial efficiency and eventual robustness.

A natural question arises whether the non-monotone condition should use the most recent function values from all iterations or only those associated with successful iterates. Our experiments indicate that the two choices yield broadly comparable performance for the two problem sets in both deterministic and stochastic settings, therefore, we omit the results.

\begin{figure}[H]
    \centering
\includegraphics[width=0.48\linewidth]{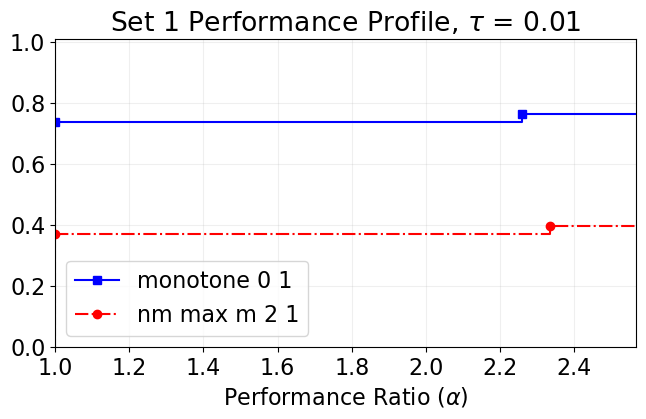}
\includegraphics[width=0.48\linewidth]{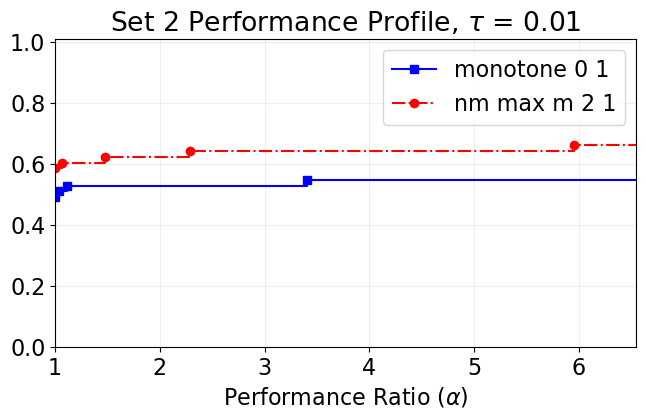}
    \caption{Performance profiles for two problem sets in the stochastic case, theoretical scenario: all stochastic function values are resampled.}
    \label{fig:sdsr}
\end{figure}

\begin{figure}[H]
    \centering
\includegraphics[width=0.48\linewidth]{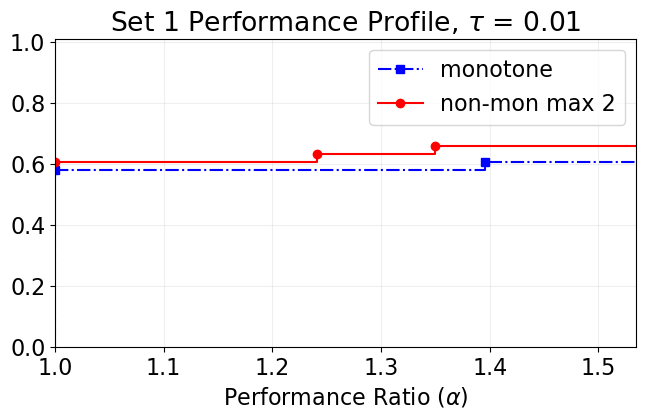}
\includegraphics[width=0.48\linewidth]{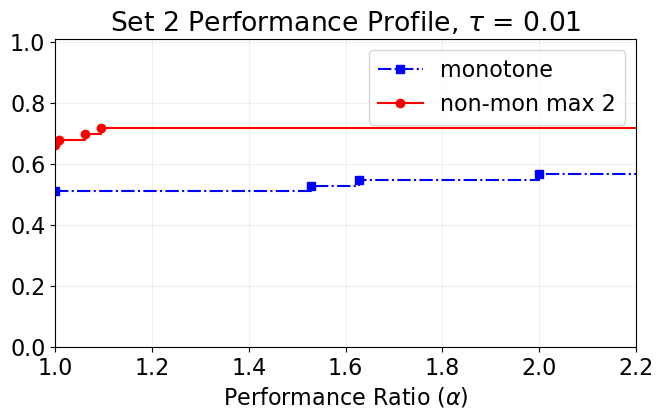}
    \caption{Performance profiles for two problem sets in stochastic case, practical scenario: only sample the value of stochastic function at $X_k^d$.}
    \label{fig:sds}
\end{figure}


Overall, the stochastic results indicate that the max-$2$ non-monotone method performs consistently well on problems with negative curvature, but resampling all historical function values can offset its benefits when stochastic evaluations are expensive.




\section{Conclusions and future work}

We studied direct-search methods in which a trial point is compared with the largest objective value among the $M$ most recent distinct iterates, rather than with the current value alone. This acceptance rule permits temporary objective increases while retaining a quantified decrease relative to the recent function values. We developed complexity analyses for three settings: deterministic complete polling, deterministic probabilistic descent, and probabilistic descent with stochastic function estimates. Together, these results provide a comprehensive complexity theory for non-monotone derivative-free optimization.
We also conducted thorough experiments to compare max-$M$ non-monotone direct search with the monotone method. Our experiment offers insights into the advantages of the practical choice $M=2$ and the option of using historical values instead of resampling. 

While both max-$M$ non-monotone and monotone direct search converge and have competitive empirical performance, 
future work could identify the problem geometries (such as suitable classes of curved valleys) for which non-monotone acceptance has a provable advantage. A non-monotone technique using the averages of objective values at recent iterates could also be an interesting topic. 

On the other hand, we expect that merit function design developed in this paper can be adapted to establish complexity bounds for line-search and trust-region methods. In a line-search framework, the usual monotone sufficient-decrease condition could be replaced by one based on the maximum objective value over a finite window of recent iterates. The analysis would then combine probabilistically accurate function and gradient estimates with a Lyapunov function that accounts for all the recent objective values. For trust-region methods, the actual reduction in the acceptance ratio could analogously be measured from the largest recent objective value. Developing such line-search and trust-region extensions, and determining whether non-monotonicity yields similar practical gains in those settings, are promising directions for future work.


\newpage
\bibliographystyle{siam}
\bibliography{reference.bib}

\begin{appendices}
\section*{Appendix }

\section{An alternative merit function for Algorithm~\ref{algo1}}\label{app:merit}
This appendix presents the alternative merit function $\Psi_k$ mentioned in Section~\ref{sec:merit_functions}. The first lemma below shows that the maximum function values of $M$ successful iterates is monotonically increasing when the index is shifted to more recent iterates. This property is helpful later to analyze the merit function. 
\begin{lemma}[Monotonicity]\label{AMon}
In Algorithm~\ref{algo1}, for each $j\geq0$, we have
\begin{align*}
\max\left(f(X_{succ,k}^{j+M-1}),\ldots,f(X_{succ,k}^{j+1}),f(X_{succ,k}^{j})\right)\leq\max\left(f(X_{succ,k}^{j+M}),\ldots,f(X_{succ,k}^{j+2}),f(X_{succ,k}^{j+1})\right).
\end{align*}
\end{lemma}

\begin{proof}
Fix any integer $j\geq 0$. We consider two cases.

First, suppose that there are at least $j+1$ successful iterations from
iteration $0$ through iteration $k$. Then the $(j+1)$-st successful
iteration in reverse order exists, and the successful-iteration condition
gives
\begin{align}
f(X_{\mathrm{succ},k}^{j})
&\leq
\max_{1\leq i\leq M}
f(X_{\mathrm{succ},k}^{j+i})
-c\left(\Delta_{\mathrm{succ},k}^{j+1}\right)^2 \leq
\max_{1\leq i\leq M}
f(X_{\mathrm{succ},k}^{j+i}).
\label{la11}
\end{align}
Moreover,
\begin{align}
\max_{1\leq i\leq M-1}
f(X_{\mathrm{succ},k}^{j+i})
\leq
\max_{1\leq i\leq M}
f(X_{\mathrm{succ},k}^{j+i}).
\label{la12}
\end{align}
Combining \eqref{la11} and \eqref{la12}, we obtain the desired result.

Now suppose that there are at most $j$ successful iterations from
iteration $0$ through iteration $k$. By the convention used to define
the reverse sequence beyond the available successful iterations,
\[
X_{\mathrm{succ},k}^{j+i}
=
X_{0},
\qquad i=1,\ldots,M.
\]
Consequently,
\[
\max_{0\leq i\leq M-1}
f(X_{\mathrm{succ},k}^{j+i})
=
f(X_{0})
=
\max_{1\leq i\leq M}
f(X_{\mathrm{succ},k}^{j+i}).
\]
\end{proof}

In the second lemma below, we show a lower bound for the smallest stepsize among the previous $M$ successful iterations in term of the current stepsize. 

\begin{lemma}\label{lemma2.7}
In Algorithm~\ref{algo1}, we have
\begin{align*}
\min_{0\leq i\leq M-1}\left\{\left(\Delta_{succ,k}^{i}\right)^2\right\}\geq\frac{\Delta_{k}^{2}}{\gamma^{2M-2}}.
\end{align*}
\end{lemma}

\begin{proof}
There are at most $M-1$ successes between the past $M-1$ successful iterations in reverse order at iteration $k$ and the $k$-th iteration. Recall that if the $i$-th success in reverse order at iteration $k$ does not exist, then $\Delta_{succ,k}^{i}=\Delta_{succ,k}^{i-1}$. Therefore, from the stepsize mechanism of Algorithm~\ref{algo1}, we have for each $0\leq i\leq M-1$
\begin{align*}
\Delta_{k}^{2}\leq\gamma^{2M-2}\left(\Delta_{succ,k}^{i}\right)^2.
\end{align*}
Therefore, we have
\begin{align*}
\Delta_{k}^{2}\leq\gamma^{2M-2}\min_{0\leq i\leq M-1}\left\{\left(\Delta_{succ,k}^{i}\right)^2\right\},
\end{align*}
which proves the result after rearranging.

\end{proof}
Let us denote the following merit function for Algorithm~\ref{algo1}: 
\begin{align*}
\Psi_{k}= &\sum_{i=0}^{M-1}\max\left(f(X_{succ,k}^{i+M-1}),\ldots,f(X_{succ,k}^{i+1}),f(X_{succ,k}^{i})\right)- M f^{*}+\eta\Delta_{k}^2 \\
& = \psi_k - M f^{*}+\eta\Delta_{k}^2, \notag
\end{align*}
where $\eta>0$ is a constant. While $\Psi_{k}$ does not need to include the correcting parameters $c_i$, $i= 1, \dots, M-1$ as in $\Phi_{k}$, it instead has the sum of $M$ prior maximum function values terms. As a result, the decrease of $\Psi_{k}$ depends on $M$ sufficient decrease conditions, as demonstrated in the following lemma. 
\begin{lemma}\label{lemma2.8}
 In Algorithm~\ref{algo1}, the $M-1$ last satisfactions of the sufficient decrease condition and the satisfaction of the current sufficient decrease condition~(\ref{sdcd}) at iteration $k$ imply that
\begin{align*}
\psi_{k}-\psi_{k+1}\geq c\min_{0\leq i\leq M-1}\left\{\left(\Delta_{succ,k}^{i}\right)^2\right\}.
\end{align*}
\end{lemma}

\begin{proof}
For $0\leq j\leq M-2$, because of the $j+1$-th satisfied sufficient decrease condition in reverse order at iteration $k$, it follows that
\begin{align*}
\max\left(f(X_{succ,k}^{j+M}),\ldots,f(X_{succ,k}^{j+2}),f(X_{succ,k}^{j+1})\right)-f(X_{succ,k}^{j})\geq c\left(\Delta_{succ,k}^{j+1}\right)^2.
\end{align*}
For ease of notation, let us denote that $B_{k}=\max\left(f(X_{succ,k}^{2M-2}),\ldots,f(X_{succ,k}^{M-1})\right)$. Together with Lemma~\ref{AMon}, one has that for each $0\leq j\leq M-2$
\begin{align}
&B_{k}-f(X_{succ,k}^{j}) \notag \\
\geq&\max\left(f(X_{succ,k}^{j+M}),\ldots,f(X_{succ,k}^{j+2}),f(X_{succ,k}^{j+1})\right)-f(X_{succ,k}^{j})\geq c\left(\Delta_{succ,k}^{j+1}\right)^2. \label{cond1-app}
\end{align}

The satisfaction of the current sufficient decrease condition gives
\begin{align*}
\max\left(f(X_{succ,k}^{M-1}),\ldots,f(X_{succ,k}^{1}),f(X_{succ,k}^{0})\right)-f(X_{succ,k+1}^{0})\geq c\left(\Delta_{succ,k}^{0}\right)^2. 
\end{align*}
Then Lemma~\ref{AMon} implies that
\begin{align}
&B_{k}-f(X_{succ,k+1}^{0}) \notag \\ 
\geq&\max\left(f(X_{succ,k}^{M-1}),\ldots,f(X_{succ,k}^{1}),f(X_{succ,k}^{0})\right)-f(X_{succ,k+1}^{0})\geq c\left(\Delta_{succ,k}^{0}\right)^2. \label{cond2-app}
\end{align}

From (\ref{cond1-app}) and (\ref{cond2-app}), we obtain
\begin{align*}
\psi_{k}-\psi_{k+1}=&\sum_{i=0}^{M-1}\max\left(f(X_{succ,k}^{i+M-1}),\ldots,f(X_{succ,k}^{i})\right)-\sum_{i=0}^{M-1}\max\left(f(X_{succ,k+1}^{i+M-1}),\ldots,f(X_{succ,k+1}^{i})\right)\\
=&\sum_{i=0}^{M-1}\max\left(f(X_{succ,k}^{i+M-1}),\ldots,f(X_{succ,k}^{i})\right)-\sum_{i=1}^{M-1}\max\left(f(X_{succ,k}^{i+M-2}),\ldots,f(X_{succ,k}^{i-1})\right)\\
&-\max\left(f(X_{succ,k}^{M-2}),\ldots,f(X_{succ,k}^{0}),f(X_{succ,k+1}^{0})\right)\\
=&B_{k}-\max\left(f(X_{succ,k}^{M-2}),\ldots,f(X_{succ,k}^{0}),f(X_{succ,k+1}^{0})\right)\\
=&\min\left(B_{k}-f(X_{succ,k}^{M-2}),\ldots,B_{k}-f(X_{succ,k}^{0}),B_{k}-f(X_{succ,k+1}^{0})\right)\\
\geq&c\min_{0\leq i\leq M-1}\left\{\left(\Delta_{succ,k}^{i}\right)^2\right\}.
\end{align*}
\end{proof}


The final lemma specifies the parameter $\eta$ and shows a deterministic decrease property for the merit function as follows. 
\begin{lemma}
Let $N_{succ,k}$ denote the number of successes before iteration $k$. Suppose that $N_{succ,k}\geq M-1$. Let $\eta=\frac{c}{\gamma^{2M-2}(\gamma^2-\theta^2)}$. There exists $\nu=\frac{c(1-\theta^2)}{\gamma^{2M-2}(\gamma^2-\theta^2)}$ such that, in Algorithm~\ref{algo1}, we have
\begin{align*}
\Psi_{k}-\Psi_{k+1}\geq \nu\Delta_{k}^{2}.
\end{align*}
\end{lemma}

\begin{proof}
Let $A_{k}=\mathbf{1}\{\text{iteration $k$ is successful}\}$. Since $N_{succ,k}\geq M-1$, it follows from Lemma~\ref{lemma2.8} and then Lemma~\ref{lemma2.7} that $A_{k}=1$ implies
\begin{align*}
\Psi_{k}-\Psi_{k+1}&=\psi_{k}-\psi_{k+1}+\eta\Delta_{k}^2(1-\gamma^2)\\
&\geq c\min_{0\leq i\leq M-1}\left\{\left(\Delta_{succ,k}^{i}\right)^2\right\}+\eta\Delta_{k}^2(1-\gamma^2)\\
&\geq\Delta_{k}^2\left(\frac{c}{\gamma^{2M-2}}+\eta(1-\gamma^2)\right)\\
&=\nu\Delta_{k}^2.
\end{align*}
Also note that $A_{k}=0$ implies that
\begin{align*}
\Psi_{k}-\Psi_{k+1}&=\psi_{k}-\psi_{k+1}+\eta\Delta_{k}^2(1-\theta^2)\\
&=\eta\Delta_{k}^2(1-\theta^2)\\
&=\nu\Delta_{k}^2.
\end{align*}
Therefore, we conclude that
\begin{align*}
\Psi_k-\Psi_{k+1}&=(\Psi_k-\Psi_{k+1})A_{k}+(\Psi_k-\Psi_{k+1})(1-A_{k})\\
&\geq\nu\Delta_k^2A_{k}+\nu\Delta_k^2(1-A_{k})\\
&=\nu\Delta_k^2.
\end{align*}
\end{proof}
Finally, from this descent lemma, one can apply similar stopping time arguments as in and Theorem~\ref{t3.9} to obtain a bound on $E(T_\epsilon)$. 

\section{Sampling complexity}\label{app:sampling}

In this appendix, we demonstrate how the tail-bound error in Assumption~\ref{a23} can be satisfied for a standard setting where the stochastic function oracle is estimated from a number of samples. 
In more detail, we suppose that at iteration $k$, the oracle $\widehat f_k(x)$ estimates the function value at point $x$ using $N_k$ independent samples as follows
\begin{align*}
\widehat f_k(x)=\frac1{N_k}\sum_{j=1}^{N_k}F(x,\zeta_{k,j}),
\end{align*}
where $E[F(x,\zeta)]=f(x)$ and $F(x,\zeta)-f(x)$ is sub-Gaussian with variance proxy $\sigma^2$. This estimation is independent of the filtration $F_{k+1/2}$. 
\begin{lemma}[A sufficient batch size]\label{prop:batch}
Let us compute the stochastic oracles in Algorithm~\ref{algo2} as described. 
The second condition of Assumption~\ref{a23} holds if
\begin{align}\label{sample-size-tail}
N_k\geq\frac{\sigma^2\left(M+1\right)^2}{\beta^2\Delta_k^4}.
\end{align}
Thus $N_k=\mathcal O(\Delta_k^{-4})$ is sufficient to guarantee the satisfaction of this condition.
\end{lemma}

\begin{proof}
Let us denote $a_0 = F(X_{k}), a_i = F(X_{succ,k}^{i}) $ for $ i \in \{1, \dots, M-1\},  a_d = F(X^{d}_{k})$, 
and $e_0, e_i$'s$, e_d$ be the corresponding errors when estimating these quantities. 
Hence 
\begin{align*}
    \varepsilon_{k}&=G_{k}-H_{k} = [\max_i(a_i + e_i) - (a_d + e_d)] - [\max_i a_i - a_d] = \max_i(a_i +e_i) - \max_i (a_i) - e_d. 
\end{align*}
And we have 
\begin{align*}
    |\varepsilon_{k}|&\leq  |\max_i(a_i +e_i) - \max_i (a_i)| + |e_d| \leq  \max_i|e_i| + |e_d| < \sum_i|e_i| + |e_d|. 
\end{align*}

As the $e_0, e_i$'s$, e_d$ are estimated with $N_k$ independent samples and each sample has Gaussian error with variance proxy $\sigma^2$, the errors $e_0, e_i$'s$, e_d$ are sub-Gaussian with variance proxy $\sigma^2/N_k$. This implies the bound 
\begin{align}\label{eq:expected-trial-error}
E\left[|e_i|\mid\mathcal F_{k+1/2}\right]
\leq \sqrt{E[e_i^2\mid\mathcal F_{k+1/2}]}
\leq \frac{\sigma}{\sqrt{N_k}}.
\end{align}
As a result 
\begin{align*}
E[|\varepsilon_k|\mid\mathcal F_{k+1/2}]
&\leq \sum_i E[|e_i|\mid\mathcal F_{k+1/2}] + E[|e_d|\mid\mathcal F_{k+1/2}]
\leq\frac{(M+1)\sigma}{\sqrt{N_k}} \leq \beta\Delta_k^2,
\end{align*}
where the last inequality follows from the sample size condition. Markov's inequality gives, for every $b>0$,
\begin{align*}
P\left(|\varepsilon_k|\geq b\mid\mathcal F_{k+1/2}\right)
&\leq \frac{E[|\varepsilon_k|\mid\mathcal F_{k+1/2}]}{b}\leq \frac{\beta\Delta_k^2}{b}.
\end{align*}
Since $\{\varepsilon_k\geq b\}\subseteq\{|\varepsilon_k|\geq b\}$, this proves the second condition of Assumption~\ref{a23}. 
\end{proof}

\section{Problem Set Details}\label{sec:app_exp}
In the experiment, we use two problem sets from the CUTEst collection~\cite{NGould_DOrban_PToint_2015}.
We describe the problem names and their corresponding dimensions below. 

\begin{table}[H]
\begin{tabular}{|l|c|l|c|l|c|}
\hline
Problems & {Dimensions} & Problems  & {Dimensions} & Problems & {Dimensions} \\
\hline
ARGLINA  & 10                             & ARGTRIGLS & 10                             & ARWHEAD  & 100                            \\
BDEXP    & 100                            & BOXPOWER  & 10                             & BROWNAL  & 10                             \\
COSINE   & 10                             & CURLY10   & 100                            & DIXON3DQ & 10                             \\
DQRTIC   & 10                             & ENGVAL1   & 100                            & EXTROSNB & 10                             \\
FLETBV3M & 10                             & FLETCBV3  & 10                             & FLETCHBV & 10                             \\
FLETCHCR & 10                             & FREUROTH  & 10                             & INDEFM   & 10                             \\
MANCINO  & 10                             & MOREBV    & 10                             & NONCVXU2 & 10                             \\
NONCVXUN & 10                             & NONDIA    & 10                             & NONDQUAR & 100                            \\
PENALTY2 & 10                             & POWER     & 10                             & QING     & 5                              \\
QUARTC   & 100                            & SENSORS   & 10                             & SINQUAD  & 100                            \\
SCOSINE  & 10                             & SCURLY10  & 10                             & SPARSINE & 10                             \\
SPARSQUR & 10                             & SSBRYBND  & 10                             & TRIDIA   & 10                             \\
TRIGON1  & 10                             & TOINTGSS  & 10                             &          & \\
\hline
\end{tabular}
\caption{Problem set 1.}
\end{table}

\begin{table}[H]
\begin{tabular}{|l|c|l|c|l|c|}
\hline
Problems & {Dimensions} & Problems & {Dimensions} & Problems & {Dimensions} \\
\hline
ALLINITU & 4                              & BARD     & 3                              & BIGGS6   & 6                              \\
BOX3     & 3                              & BROYDN7D & 10                             & BRYBND   & 10                             \\
CUBE     & 2                              & DENSCHND & 3                              & DENSCHNE & 3                              \\
DIXMAANA & 15                             & DIXMAANB & 15                             & DIXMAANC & 15                             \\
DIXMAAND & 15                             & DIXMAANE & 15                             & DIXMAANF & 15                             \\
DIXMAANG & 15                             & DIXMAANH & 15                             & DIXMAANI & 15                             \\
DIXMAANJ & 15                             & DIXMAANK & 15                             & DIXMAANL & 15                             \\
ENGVAL2  & 3                              & ERRINROS & 25                             & EXPFIT   & 2                              \\
FMINSURF & 16                             & GROWTHLS & 3                              & GULF     & 3                              \\
HAIRY    & 2                              & HATFLDD  & 3                              & HATFLDE  & 3                              \\
HEART6LS & 6                              & HEART8LS & 8                              & HELIX    & 3                              \\
HIELOW   & 3                              & HIMMELBB & 2                              & HIMMELBG & 2                              \\
HUMPS    & 2                              & KOWOSB   & 4                              & LOGHAIRY & 2                              \\
MARATOSB & 2                              & MEYER3   & 3                              & MSQRTALS & 4                              \\
MSQRTBLS & 9                              & OSBORNEA & 5                              & OSBORNEB & 11                             \\
PENALTY3 & 50                             & SNAIL    & 2                              & SPMSRTLS & 28                             \\
STRATEC  & 10                             & VIBRBEAM & 8                              & WATSON   & 12                             \\
WOODS    & 4                              & YFITU    & 3                              &          & \\ 
\hline
\end{tabular}
\caption{Problem set 2.}
\end{table}

\end{appendices}

\end{document}